\documentclass[12pt]{amsart}
\usepackage{}

\usepackage{amsmath}
\usepackage{amsfonts}
\usepackage{amssymb}
\usepackage[all]{xy}           
\usepackage{bbm}
\usepackage{bbding}
\usepackage{txfonts}
\usepackage{amscd}
\usepackage{tikz}
\usetikzlibrary{matrix}
\usepackage[shortlabels]{enumitem}

\usepackage{ifpdf}
\ifpdf
\usepackage[colorlinks,final,backref=page,hyperindex]{hyperref}
\else
\usepackage[colorlinks,final,backref=page,hyperindex,hypertex]{hyperref}
\fi
\usepackage{tikz}
\usepackage[active]{srcltx}

\makeatletter

\newtheorem{thm}{Theorem}[section]

\newtheorem{pro}[thm]{Proposition}
\newtheorem{ex}[thm]{Example}
\newtheorem{rmk}[thm]{Remark}
\newtheorem{defi}[thm]{Definition}

\newcommand {\emptycomment}[1]{}

\newcommand{\be }{\begin{equation}}
\newcommand{\ee }{\end{equation}}

\newcommand{\g}{\mathfrak g}

\newcommand{\Comp}{\mathbb C}

\newcommand{\frkd}{\mathfrak d}

\newcommand{\br}[1]{   [ \cdot,    \cdot  ]   }

\newcommand{\gl}{\mathfrak {gl}}

\newcommand{\ad}{\mathrm{ad}}
\newcommand{\ada}{\mathfrak{ad}}

\usepackage{lipsum}

\newcommand\blfootnote[1]{%
  \begingroup
  \renewcommand\thefootnote{}\footnote{#1}%
  \addtocounter{footnote}{-1}%
  \endgroup
}
\begin{document}

\setlength{\baselineskip}{1.2\baselineskip}
\title[Manin triples associated to $n$-Lie bialgebras]
{Manin triples associated to $n$-Lie bialgebras}
\author{Ying Chen}
\address{
School of Mathematical Sciences  \\
Zhejiang Normal University\\
Jinhua 321004\\
China}
\email{yingchen@zjnu.edu.cn}

\author{Chuangchuang Kang}
\address{
School of Mathematical Sciences  \\
Zhejiang Normal University\\
Jinhua 321004 \\
China}
\email{kangcc@zjnu.edu.cn}

\author{Jiafeng L\"u}
\address{
School of Mathematical Sciences    \\
Zhejiang Normal University\\
Jinhua 321004              \\
China}
\email{jiafenglv@zjnu.edu.cn}

\author{Shizhuo Yu}
\address{
School of Mathematical Sciences and LPMC    \\
Nankai University \\
Tianjin 300317              \\
China}
\email{yusz@nankai.edu.cn}

\blfootnote{*Corresponding Author: Chuangchuang Kang. Email: kangcc@zjnu.edu.cn.}

\dedicatory{\it{Dedicated to the Memory of Professor Yuri I. Manin (1937-2023)}}
\begin{abstract}
In this paper, we study the Manin triples associated to $n$-Lie bialgebras. We introduce the concept of operad matrices for $n$-Lie bialgebras. In particular, by studying  a special case of operad matrices, it leads to the notion of local cocycle $n$-Lie bialgebras. Furthermore, we establish a one-to-one correspondence between the double of $n$-Lie bialgebras and Manin triples of $n$-Lie algebras.
\end{abstract}





\maketitle
\allowdisplaybreaks

\section{Introduction}\label{sec:intr}


Manin triples structures naturally induce a class of quasi-triangular $r$-matrix, which provide a class of Poisson manifolds and Poisson homogeneous spaces in Lie theory \cite{Semenov-Tian-Shansky}. In 1983, Drinfeld \cite{Drinfeld} introduced the notion of Lie bialgebras, which is well established as the infinitesimalisation of quantum groups \cite{Manin1}.
A Lie bialgebra consists of a Lie algebra $\g$ and a compatible Lie cobracket $\delta_\g$, such that the cobracket induces a Lie bracket on the dual space $\g^*$ and satisfies the $1$-cocycle condition \cite{Drinfeld,Quantum2}. Moreover, Lie bialgebras exponentiate to Poisson-Lie groups, which has attracted considerable interest from Poisson and symplectic geometers \cite{Quantum}.
In fact, let $(\g,\delta_\g)$ be a Lie bialgebra, then there exists a canonical Lie bialgebra structure on $\g\oplus \g^*$ induced by Manin triples of Lie algebras. The Lie bialgebra on $\g\oplus \g^*$ is called the double Lie bialgebra of $\g$, and it can be used to construct Poisson manifolds \cite{Lu-1}.
The usual interpretation of the $1$-cocycle condition for Lie bialgebras is that $\delta_\g$ is a $1$-cocycle of $\g$ associated to the representation $\ad\otimes1+1\otimes\ad$ on the tensor space $\g\otimes\g$. Another way to interpret the $1$-cocycle condition is to decompose $\delta_\g$ into $\delta_\g^1$ and $\delta_\g^2$. Here,
$\delta_\g^1$ and $\delta_\g^2$ are $1$-cocycles of $\g$ associated to $\ad\otimes1$ and $1\otimes\ad$, respectively, satisfying a compatibility condition \cite{Bai}. This equivalent interpretation leads to the local cocycle condition and can be understood by an operadic point of view \cite{Loday}. Inspired by this,
we extend such structures to the $n$-Lie bialgebras, which are operad matrices of $n$-Lie bialgebras. We also establish a one-to-one correspondence between the double of $n$-Lie bialgebras and Manin triples of $n$-Lie algebras, and get some applications.

In 1985, Filippov \cite{Filippov} introduced the definition of $n$-Lie algebras (Filippov algebras). His paper considered $n$-ary multi-linear and skew-symmetric operation that satisfy the generalized Jacobi identity, which appear in many fields in mathematics and mathematical physics, particularly, $3$-Lie algebras play an important role in string theory \cite{$n$-ary,Bagger,Bagger2,Drinfeld2,Gustavsson,Matsuo,Matsuo2}.
In 2016, the authors in \cite{Bai} introduced two types of $3$-Lie bialgebras, whose compatibility conditions are given by local cocycles and double constructions respectively. The notion of 3-Lie classical Yang-Baxter equation (3-Lie CYBE) is derived from local cocycle 3-Lie bialgebras, and the solutions give rise to coboundary local cocycle 3-Lie bialgebras. Meanwhile, $3$-pre-Lie algebras give rise to solutions of the $3$-Lie CYBE. In 2017, the authors in \cite{Bai2} classified the double construction 3-Lie bialgebras for complex 3-Lie algebras in dimensions 3 and 4 and provided the corresponding pseudo-metric 3-Lie algebras of dimension 8. In \cite{Bai-R}, $n$-Lie coalgebras with rank $r$ are defined, their structures are discussed, and $n$-Lie bialgebras are introduced and their structures are investigated.

However,  there is currently unknown coboundary theory
or structure for the double space  $\g\oplus \g^*$ of $n$-Lie bialgebras, where $\g$ is an $n$-Lie algebra.
Inspired by the notations of local cocycle 3-Lie bialgebras and double
construction 3-Lie bialgebras, it is natural to consider the Manin triples associated to $n$-Lie bialgebras.

The paper is organized as follows. Theorem \ref{main} in Section \ref{sec:double} is the main result of the paper and Section \ref{sec:preliminary} as well as Section \ref{sec:n-Lie-bi} are the preparations for it. In Section \ref{sec:preliminary}, we introduce some concepts and known results about $n$-Lie algebras that will be used later. In Section \ref{sec:n-Lie-bi}, we summarize the coboundary theory of $n$-Lie algebras. We also show that each $n$-Lie bialgebra must have a dual $n$-Lie bialgebra whose dual is the $n$-Lie bialgebra itself. In Section \ref{sec:double}, we define operad matrices of $n$-Lie bialgebras and generalize local cocycle $3$-Lie bialgebras introduced by \cite{Bai} to local cocycle $n$-Lie bialgebras (Proposition \ref{defi:local-coc}). We also establish a one-to-one correspondence between the double of $n$-Lie bialgebras and  Manin triples of $n$-Lie algebras (Theorem \ref{main}).

Throughout this paper, all algebras are finite-dimensional and over a field $F$ of
characteristic zero.
\section{Preliminary results on $n$-Lie algebras}\label{sec:preliminary}
In this section, we give some preliminaries and basic results on $n$-Lie algebras from \cite{Filippov}.
\begin{defi}\label{defi:n-Lie-alg}
An {\bf $n$-Lie algebra} is a vector space $\g$ with a skew-symmetric $n$-linear map $[\cdot,\cdots,\cdot]:\otimes^n\g\rightarrow \g$ such that the following Filippov-Jacobi identity holds, for all $x_1,\cdots,x_{n-1},y_1,\cdots,y_{n} \in \g$,
\begin{equation}\label{eq:FJi}
  [x_1,\cdots,x_{n-1},[y_1,\cdots,y_n]]=\sum\limits_{i=1}^n[y_1,\cdots,y_{i-1},[x_1,\cdots,x_{n-1},y_i],y_{i+1},\cdots,y_n].
\end{equation}

\end{defi}

The Filippov-Jacobi identity can be described  as for any $X=(x_1,\cdots,x_{n-1}) \in \wedge^{n-1}\g$, the operator
\begin{equation*}
  \ad(X):\g\rightarrow\g,\quad\ad(X)(y):=[x_1,\cdots,x_{n-1},y],\quad\forall~ y \in \g,
\end{equation*}
is a derivation in the sense that
\begin{equation}\label{eq:FJi-deri}
  \ad(X)([y_1,\cdots,y_n])=\sum\limits_{i=1}^n[y_1,\cdots,y_{i-1},\ad(X)(y_i),y_{i+1},\cdots,y_n].
\end{equation}

\begin{defi}\label{defi:$n$-Lie-alg-rep}
A {\bf representation} of an $n$-Lie algebra $\g$ on a vector space M is a linear map $\rho:\wedge^{n-1}\g\rightarrow \gl(M)$ satisfying
\begin{align}
\rho([x_1,\cdots,x_n],y_1,\cdots,y_{n-2}) &= \sum\limits_{i=1}^n(-1)^{n-i}\rho(x_1,\cdots,\hat{x_i},\cdots,x_n)\rho(x_i,y_1,\cdots,y_{n-2}), \label{eq:rep-1}\\
[ \rho(x_1,\cdots,x_{n-1}),\rho(y_1,\cdots,y_{n-1})] &= \sum\limits_{i=1}^{n-1}\rho(y_1,\cdots,y_{i-1},[x_1,\cdots,x_{n-1},y_i],y_{i+1},\cdots,y_{n-1}),
\label{eq:rep-2}
\end{align}
for all $x_1,\cdots,x_n,y_1,\cdots,y_{n-1} \in \g$, where $\hat{x_i}$ means that the element $x_i$ is omitted.

\end{defi}
We denote the representation by the pair $(M,\rho)$, and say that $M$ is a $\g$-module as well. If $\rho=\ad:\wedge^{n-1}\g\rightarrow \gl(\g)$ is given by
$$
\ad_{x_1,\cdots,x_{n-1}}(x_n)=[x_1,\cdots,x_{n-1},x_n],\quad\forall~x_1,\cdots,x_n\in \g,
$$
then the pair $(\g,\ad)$ is a $\g$-module and is called the adjoint module of $\g$.


\begin{pro}
Let $\g$ be an $n$-Lie algebra, then for all $x_1,\cdots,x_{n-1},y_1,\cdots,y_{n}\in \g$,
\begin{eqnarray}\label{eq:FJi-biyao}
  && \sum\limits_{i=1}^{n-1}[y_1,\cdots,y_{i-1},[x_1,\cdots,x_{n-1},y_i],y_{i+1},\cdots,y_{n-1},y_n] \\
\nonumber  &+&\sum\limits_{j=1}^{n-1}[x_1,\cdots,x_{j-1},[y_1,\cdots,y_{n-1},x_j],x_{j+1},\cdots,x_{n-1},y_n]=0.
\end{eqnarray}
\end{pro}

\begin{proof}
By \eqref{eq:FJi}, for all $x_1,\cdots,x_{n-1},y_1,\cdots,y_{n}\in \g$,
\begin{eqnarray*}
  [x_1,\cdots,x_{n-1},[y_1,\cdots,y_n]] &=& \sum\limits_{i=1}^{n-1}[y_1,\cdots,y_{i-1},[x_1,\cdots,x_{n-1},y_i],y_{i+1},\cdots,y_n]\\
  &+&[y_1,\cdots,y_{n-1},[x_1,\cdots,x_{n-1},y_n]] \\
  &=& \sum\limits_{i=1}^{n-1}[y_1,\cdots,y_{i-1},[x_1,\cdots,x_{n-1},y_i],y_{i+1},\cdots,y_{n-1},y_n]\\
  &+&\sum\limits_{j=1}^{n-1}[x_1,\cdots,x_{j-1},[y_1,\cdots,y_{n-1},x_j],x_{j+1},\cdots,x_{n-1},y_n]\\
  &+&[x_1,\cdots,x_{n-1},[y_1,\cdots,y_n]],
\end{eqnarray*}
implying \eqref{eq:FJi-biyao}.
\end{proof}
Let $V$ be a vector space, $V^*$ be the dual space of $V$. For each positive integer $k$, we identify the tensor product $\otimes^k V$ with the space of multi-linear
maps from $\underbrace{V^*\times \cdots \times V^*}_{k-times}\rightarrow \Comp$, such that
\begin{equation}\label{eq:action}
\langle\xi_1\otimes \cdots\otimes \xi_k,v_1\otimes\cdots\otimes v_k\rangle=\langle\xi_1,v_1\rangle\cdots\langle\xi_k,v_k\rangle,\quad \forall ~\xi_1,\cdots,\xi_k\in V^*,v_1,\cdots,v_k\in V,
\end{equation}
where $\langle \xi_i,v_i\rangle=\xi_i(v_i)$, $1\leq i\leq k$. For $v_1,\cdots,v_k\in V$, define that
$$
v_1\wedge v_2 \wedge \cdots\wedge v_k=\sum_{\sigma\in S_k}sgn(\sigma)v_{\sigma(1)}\otimes v_{\sigma(2)}\cdots v_{\sigma(k)}\in \wedge^kV\subset \otimes^kV.
$$

\section{$n$-Lie bialgebras}\label{sec:n-Lie-bi}

In this section, we first introduce the coboundary theory of $n$-Lie algebras, then we give the notions of $n$-Lie bialgebras and the coadjoint representation of $n$-Lie algebras.
Finally, in Proposition \ref {pro:$n$-Lie-bialg-dual}, we  show that each $n$-Lie bialgebra must have a dual $n$-Lie bialgebra whose dual is the $n$-Lie bialgebra itself.

\subsection{$n$-Lie algebra cohomology}

Let $\g$ be an $n$-Lie algebra over $F$, for all $x_1,\cdots,x_{n-1}\in \g$, denote the elements $X=x_1\wedge x_2\wedge\cdots\wedge x_{n-1}$ in $\wedge^{n-1}\g$ by $X=(x_1,\cdots,x_{n-1})$.
Let $(M, \rho)$ be a $\g$-module. For any $X=(x_1, \cdots, x_{n-1})\in
\wedge^{n-1}\g$, and $a\in M$, the action of $\g$ on $M$
is denoted by $X(x_1, \cdots, x_{n-1})_\cdot a$ or $X_\cdot a$.


We know that if $\g$ is an $n$-Lie algebra,  then $(\otimes^p \g, \ad^{(p)})$ is a $\g$-module for any positive integer $p$, where for all $y_1\otimes\cdots\otimes y_p \in\otimes^p\g$, $x_1, \cdots x_{n-1}\in \g$,
\begin{equation}\label{eq:ad-p}
  X_\cdot (y_1\otimes\cdots\otimes y_p) := \ad_{x_1,\cdots,x_{n-1}}^{(p)}(y_1\otimes \cdots\otimes y_p)
   = \sum_{i=1}^p y_1\otimes \cdots\otimes\ad_{x_1,\cdots,x_{n-1}}(y_i)\otimes\cdots \otimes y_p.
\end{equation}
\begin{defi} \cite{Casas}
Let $\g$ be an $n$-Lie algebra and $(M,\rho)$ be a representation of $\g$. {\bf $k$-cochains} on $\g$ with values in $M$ is the set
\begin{equation*}\label{defi:k-cochains}
  C^k(\g;M):=~\{~linear~maps~u:\underbrace{\wedge^{n-1}\g\otimes\cdots \otimes\wedge^{n-1}\g}_{k-1 -times}\wedge\g\rightarrow M~\}.
\end{equation*}
\end{defi}

\begin{defi} \cite{Casas}
Let $\g$ be an $n$-Lie algebra and $(M,\rho)$ be a representation of $\g$. The coboundary operator of a $k$-cochain $u$ on $\g$ with values in $M$ is the map $\delta_\rho :C^k(\g;M)\rightarrow C^{k+1}(\g;M)$,  such that for all $X_i=(x_1^i,\cdots , x_{n-1}^i) \in \wedge^{n-1}\g, 1\leq i\leq k$, $z \in \g$,
\begin{align*}\label{defi:k-coboundary}
  &(\delta_\rho u)(X_1,X_2,\cdots,X_k,z) \\
  =&\sum\limits_{i=1}^{k}(-1)^{i+1}\rho(x_1^i,\cdots,x_{n-1}^i)u(X_1,\cdots,\hat{X_i},\cdots,X_k,z) \\
  +&\sum\limits_{i=1}^{n-1}(-1)^{n+k-i+1}\rho(x_1^k,\cdots,\hat{x}_i^k,\cdots,x_{n-1}^k,z)u(X_1,\cdots,X_{k-1},\hat{X_k},x_i^k) \\
  +&\sum\limits_{i=1}^{k}(-1)^{i}u(X_1,\cdots,\hat{X_i},\cdots,X_k,[x_1^i,\cdots,x_{n-1}^i,z])  \\
  +&\sum\limits_{1\leq i\leq j}^{k}(-1)^{i}u(X_1,\cdots,\hat{X_i},\cdots,X_{j-1},\sum\limits_{m=1}^{n-1}[x_1^i,\cdots,x_{n-1}^i,x_m^j]\wedge x_1^j\wedge\cdots\wedge\hat{x}_m^j\wedge\cdots\wedge x_{n-1}^j,\cdots,X_k,z),
\end{align*}
where $\hat{X_i}$ indicates that the element $X_i$ is omitted.
\end{defi}

\begin{pro}\label{pro:complex-$n$-Lie}\cite{Casas}
For any $k$-cochain $u$, $k\geq1$, $\delta_\rho(\delta_\rho u)=0$.
\end{pro}

%

\begin{defi}\label{defi:k-cocycle}\cite{Casas}
A $k$-cochain $u$ is called a {\bf $k$-cocycle} if it satisfies
\begin{equation*}
  \delta_\rho u=0.
\end{equation*}
A $k$-cochain $u$ is called a {\bf $k$-coboundary} $(k\geq2)$ if there exists a $(k-1)$-cochain $v$, such that
\begin{equation*}
  u=\delta_\rho v.
\end{equation*}
\end{defi}
By the Proposition \ref{pro:complex-$n$-Lie}, any $k$-coboundary is a $k$-cocycle. By the Definition \ref{defi:k-cocycle}, the quotient of the vector space of $k$-cocycles by the vector space of $k$-coboundaries is called the $k$-th cohomology vector space of $\g$ with values in $M$.
See \cite{Takhtajan-2} for more details.
\subsection{$n$-Lie bialgebras}
Now assume that $\g$ is an $n$-Lie algebra and $\gamma$ is a linear map from $\g$ to $\otimes^n\g$ whose transpose is denoted by $^t\gamma:\otimes^n\g^\ast\rightarrow\g^\ast$.

\begin{defi}\label{defi:$n$-Lie-bialg} {\rm{\cite{Bai-R}}}
An {\bf $n$-Lie bialgebra} is an $n$-Lie algebra $\g$ with a linear map $\gamma:\g\rightarrow\otimes^n\g$ such that
\begin{enumerate}[{\rm(i)}]
  \item $^t\gamma:\otimes^n\g^\ast\rightarrow\g^\ast$ defines an $n$-Lie bracket on $\g^\ast$, i.e., $^t\gamma$ is a skew-symmetric $n$-linear map on $\g^\ast$ satisfying the Filippov-Jacobi identity;
  \item $\gamma$ is a $1$-cocycle on $\g$ with values in $\otimes^n\g$, where $\g$ acts on $\otimes^n\g$ by the adjoint representation $\ad^{(n)}$.
\end{enumerate}
\end{defi}

Condition (i) is equivalent to the following two identities, for all $\xi_1, \cdots, \xi_{n}, \eta_1,\cdots, \eta_n\in \g^*$,
\begin{align*}
[\xi_{\sigma(1)},\cdots,\xi_{\sigma(n)}]_{\g^\ast} &=sgn(\sigma)[\xi_1,\cdots,\xi_n]_{\g^\ast},\\
[\xi_1,\cdots,\xi_{n-1},[\eta_1,\cdots,\eta_n]_{\g^\ast}]_{\g^\ast}&=
\sum\limits_{i=1}^{n}[\eta_1,\cdots,\eta_{i-1},[\xi_1,\cdots,\xi_{n-1},\eta_i]_{\g^\ast},\eta_{i+1},\cdots,\eta_n]_{\g^\ast}.
\end{align*}

Condition (ii) means that the 2-cochain $\delta\gamma$ vanishes, i.e., for all $x_1,\cdots,x_n \in \g$,
\begin{equation}\label{eq:ii'}
 \gamma([x_1,\cdots,x_n])=\sum\limits_{i=1}^{n}(-1)^{n-i}\ad^{(n)}_{x_1,\cdots,\hat{x_i},\cdots,x_n}(\gamma(x_i)).
\end{equation}
Let $[\cdot,\cdots,\cdot]_{\g^*}:\otimes^n \g^*\rightarrow \g^*$ be the $n$-Lie bracket defined by $\gamma$, that is for all $\xi_1,\cdots,\xi_n\in\g^\ast$,
\begin{equation}
  [\xi_1,\cdots,\xi_n]_{\g^\ast}=^t\gamma(\xi_1\otimes\cdots\otimes\xi_n).
\end{equation}
Then by the Definition \ref{defi:$n$-Lie-bialg}, for all $x \in \g$,
\begin{equation}\label{eq:gamma-tran}
  \langle~[\xi_1,\cdots,\xi_n]_{\g^\ast},x~\rangle = \langle~\gamma(x),\xi_1\otimes\cdots\otimes\xi_n~\rangle.
\end{equation}

Let $1:\g\rightarrow \g$ be the identity map,  an alternate way of \eqref{eq:ii'} is
\begin{align*}\label{eq:expandii'}
  \langle~[\xi_1,\cdots,\xi_n]_{\g^\ast},[x_1,\cdots,x_n]~\rangle
  =\sum\limits_{i=1}^{n}\langle~\xi_1\otimes\cdots\otimes\xi_n,\sum_{j=1}^n (1\otimes^{j-1} \ad_{x_1,\cdots,\hat{x_i},\cdots,x_n}\otimes^{n-j}1)\gamma(x_i)~\rangle.
\end{align*}
Using the Sweedler notation, write $\gamma(x_i)=z_1\otimes \cdots\otimes z_n$, we have
\begin{align*}
  \sum_{j=1}^n (1\otimes^{j-1} \ad_{x_1,\cdots,x_{n-1}}\otimes^{n-j}1)
  (z_1\otimes \cdots\otimes z_n)
  = \sum_{j=1}^n z_1\otimes\cdots\otimes[x_1,\cdots,x_{n-1},z_j]\otimes \cdots\otimes z_n.
\end{align*}

\subsection{The coadjoint representation of $n$-Lie algebras}
Now we introduce the definition of the coadjoint representation of $n$-Lie algebras on the dual vector space.

\begin{pro}\label{pro:co-ad}
Let $\g$ be an $n$-Lie algebra and $\g^\ast$ be the dual vector space. Set
$\ad^*=-^t\ad$, that is, for all $x_1, \cdots, x_{n-1}\in \g$,
$\ad^\ast_{x_1,\cdots,x_{n-1}}$ is the endomorphism of $\g^\ast$ satisfying
\begin{equation}\label{eq:dual-rep}
  \langle~\ad^\ast_{x_1,\cdots,x_{n-1}}\xi, x~\rangle=-\langle~\xi, \ad_{x_1,\cdots,x_{n-1}}x~\rangle,\quad\forall~x \in \g, \xi \in \g^\ast.
\end{equation}
Then $(\g^\ast,\ad^\ast)$ is a representation of $\g$ on $\g^\ast$.
\end{pro}

\begin{proof}
By \eqref{eq:FJi-biyao} and \eqref{eq:dual-rep}, for all $x_1,\cdots,x_n,~y_1,\cdots,y_{n-1},~z \in \g, ~\xi \in \g^\ast$, we have
\begin{align*}
  &\langle~\ad_{[x_1,\cdots,x_n],y_1,\cdots,y_{n-2}}^\ast\xi,z~\rangle \\
  =&(-1)^n\langle~\xi,[y_1,\cdots,y_{n-2},z,[x_1,\cdots,x_n]]~\rangle,\\
  =&(-1)^n\sum\limits_{i=1}^{n}\langle~\xi,[x_1,\cdots,x_{i-1},[y_1,\cdots,y_{n-2},z,x_i],x_{i+1},\cdots,x_n]~\rangle\\
=&(-1)^{n-i}\sum\limits_{i=1}^{n}\langle~\xi,[x_i,y_1,\cdots,y_{n-2},[x_1,\cdots,\hat{x_i},\cdots,x_n,z]]~\rangle\\
=&\langle~\sum\limits_{i=1}^{n}(-1)^{n-i}\ad_{x_1,\cdots,\hat{x_i},\cdots,x_n}^\ast\ad_{x_i,y_1,\cdots,y_{n-2}}^\ast\xi,z~\rangle.
\end{align*}
And
\begin{align*}
  &\langle~\ad_{x_1,\cdots,x_{n-1}}^\ast\ad_{y_1,\cdots,y_{n-1}}^\ast\xi,z~\rangle-\langle~\ad_{y_1,\cdots,y_{n-1}}^\ast\ad_{x_1,\cdots,x_{n-1}}^\ast\xi,z~\rangle \\
  =&\langle~\xi,[y_1,\cdots,y_{n-1},[x_1,\cdots,x_{n-1},z]]~\rangle-\langle~\xi,[x_1,\cdots,x_{n-1},[y_1,\cdots,y_{n-1},z]]~\rangle \\
  =& \langle \xi,\sum\limits_{i=1}^{n-1}[x_1,\cdots,x_{i-1},[y_1,\cdots,y_{n-1},x_i],x_{i+1},\cdots,x_{n-1},z]\rangle\\
=& \langle~\xi,-\sum\limits_{i=1}^{n-1}[y_1,\cdots,y_{i-1},[x_1,\cdots,x_{n-1},y_i],y_{i+1},\cdots,y_{n-1},z]~\rangle\\
=&\langle~\sum\limits_{i=1}^{n-1}\ad_{y_1,\cdots,y_{i-1},[x_1,\cdots,x_{n-1},y_i],y_{i+1},\cdots,y_{n-1}}^\ast\xi,z~\rangle.
\end{align*}

Then we get  for all $x_1, \cdots, x_n, y_1, \cdots, y_{n-1}\in \g$,
\begin{align*}
  \ad_{[x_1,\cdots,x_n],y_1,\cdots,y_{n-2}}^\ast &= \sum\limits_{i=1}^{n}(-1)^{n-i}\ad_{x_1,\cdots,\hat{x_i},\cdots,x_n}^\ast\ad_{x_i,y_1,\cdots,y_{n-2}}^\ast,\\
\ad_{x_1,\cdots,x_{n-1}}^\ast\ad_{y_1,\cdots,y_{n-1}}^\ast-\ad_{y_1,\cdots,y_{n-1}}^\ast\ad_{x_1,\cdots,x_{n-1}}^\ast &= \sum\limits_{i=1}^{n-1}\ad_{y_1,\cdots,y_{i-1},[x_1,\cdots,x_{n-1},y_i],y_{i+1},\cdots,y_{n-1}}^\ast.
\end{align*}
Therefore, $(\g^\ast,\ad^\ast)$ is a representation of $\g$ on $\g^\ast$.
\end{proof}

\begin{defi}\label{defi:$n$-Lie-alg-coadjoint-rep}
The representation $(\g^\ast, \ad^\ast)$  is called the {\bf coadjoint representation} of $\g$.
\end{defi}

\subsection{The dual of $n$-Lie bialgebras}

Let  $(\g^\ast, \ad^\ast)$  be the  coadjoint representation of $n$-Lie algebra $\g$.
By  \eqref{eq:action}  and \eqref{eq:dual-rep}, for all $\xi_1,\cdots,\xi_n\in \g^*$ and $y_1\otimes\cdots\otimes y_n$ in $\otimes^n\g$, we have
\begin{equation}\label{eq:dual-ten}
  \langle~\ad^{*(p)}(\xi_1\otimes\cdots\otimes\xi_n),y_1\otimes\cdots\otimes y_n~\rangle=
  -\langle~\xi_1\otimes\cdots\otimes\xi_n,\ad^{(p)}(y_1\otimes\cdots\otimes y_n)~\rangle,
\end{equation}
where $\ad^{(p)}$ is defined by \eqref{eq:ad-p}.

Let $(\g,\gamma)$ be an $n$-Lie bialgebra in the Definition \ref{defi:$n$-Lie-bialg} and $[\cdot,\cdots,\cdot]_{\g^*}:\otimes^n \g^*\rightarrow \g^*$ be the $n$-Lie bracket defined by $\gamma$.
 By \eqref{eq:ii'} and \eqref{eq:dual-ten}, for all $\xi_1,\cdots,\xi_n\in \g^*$, $x_1,\cdots,x_n\in \g$, we have
\begin{align*}
  &\langle~[\xi_1,\cdots,\xi_n]_{\g^\ast}, [x_1,\cdots,x_n]~\rangle\\
  =&\langle~\xi_1\otimes\cdots \otimes\xi_n, \gamma([x_1,\cdots,x_n])~\rangle\\
  =&\langle~\xi_1\otimes\cdots \otimes\xi_n, \sum\limits_{i=1}^{n}(-1)^{n-i}\ad^{(n)}_{x_1,\cdots,\hat{x_i},\cdots,x_n}(\gamma(x_i))~\rangle\\
  =&\sum\limits_{i,j=1}^{n}(-1)^{n-i+1}\langle~[\xi_1,\cdots,\ad_{x_1,\cdots,\hat{x_i},\cdots,x_n}^*(\xi_j),\cdots,\xi_n]_{\g^\ast},x_i~\rangle\\
  =&\sum\limits_{i,j=1}^{n}(-1)^{i+j-1}
  \langle~[\xi_1,\cdots,\xi_{i-1},\xi_{i+1},\cdots,\xi_n,\ad^\ast_{x_1,\cdots,\hat{x_j},\cdots,x_n}(\xi_i)]_{\g^\ast},x_j~\rangle.
\end{align*}

Then \eqref{eq:ii'} can be written by
\begin{equation}\label{eq:$1$-cocycle-equ}
  \langle~[\xi_1,\cdots,\xi_n]_{\g^\ast}, [x_1,\cdots,x_n]~\rangle
  =\sum\limits_{i,j=1}^{n}(-1)^{i+j-1}
  \langle~[\xi_1,\cdots,\xi_{i-1},\xi_{i+1},\cdots,\xi_n,\ad^\ast_{x_1,\cdots,\hat{x_j},\cdots,x_n}(\xi_i)]_{\g^\ast},x_j~\rangle.
\end{equation}
For example, in \cite{Bai}, the $1$-cocycle condition of Lie bialgebras can be written by
\begin{eqnarray*}
  \langle~ [\xi_1,\xi_2]_{\g^*},[x_1,x_2]\rangle
  &=&-\langle~[\xi_2,\ad_{x_2}^*(\xi_1)],x_1~\rangle
  +\langle~[\xi_1,\ad_{x_2}^*(\xi_2)],x_1~\rangle\\
  &+&\langle~[\xi_2,\ad_{x_1}^*(\xi_1)],x_2~\rangle
  -\langle~[\xi_1,\ad_{x_1}^*(\xi_2)],x_2~\rangle.
\end{eqnarray*}

 For any $n$-Lie bialgebra $(\g,\gamma)$, denotes $\ada:\wedge^{n-1}\g^*\rightarrow End(\g^*)$ the adjoint representation of the $n$-Lie algebra $\g^*$.
Then, by the Proposition \ref{pro:co-ad}, $(\g,\ada^*)$ is the coadjoint representation of $\g^\ast$ on $\g$ and we have
\begin{equation}\label{eq:$1$-cocycle-equ2}
  \langle~[\xi_1,\cdots,\xi_n]_{\g^\ast}, [x_1,\cdots,x_n]~\rangle =\sum\limits_{i,j=1}^{n}(-1)^{i+j}\langle~\ad^\ast_{x_1,\cdots,\hat{x_j},\cdots,x_n}(\xi_i)
  ,\ada^\ast_{\xi_1,\cdots,\hat{\xi_i},\cdots,\xi_n}(x_j)~\rangle.
\end{equation}
\begin{pro}\label{pro:$n$-Lie-bialg-dual}
Let $(\g,\gamma)$ be an $n$-Lie bialgebra,  $[\cdot,\cdots,\cdot]_{\g^*}:\wedge^n\g^*\rightarrow \g^*$ be the $n$-Lie bracket induced by $\gamma$ and $\mu$ be the $n$-Lie bracket of $\g$. Then $(\g^\ast, ^t\mu)$ is an $n$-Lie bialgebra.

\end{pro}
\begin{proof}
For all $\xi_1,\cdots,\xi_n\in \g^\ast$, $x_1,\cdots,x_n\in \g$, we have
\begin{equation}\label{eq:mu-tran}
  \langle~[\xi_1,\cdots,\xi_n]_{\g^\ast}, [x_1,\cdots,x_n]~\rangle=\langle~^t\mu([\xi_1,\cdots,\xi_n]_{\g^\ast}), x_1\otimes \cdots\otimes x_n~\rangle.
\end{equation}
By \eqref{eq:$1$-cocycle-equ2} and \eqref{eq:mu-tran}, we have
\begin{eqnarray*}
  && \langle~^t\mu([\xi_1,\cdots,\xi_n]_{\g^\ast}), x_1\otimes \cdots\otimes x_n~\rangle \\
  &=& \sum\limits_{i,j=1}^{n}(-1)^{i+j+1}\langle~\xi_i,[x_1,\cdots,x_{j-1},\hat{x_j},x_{j+1},\cdots,x_n,
  \ada^\ast_{\xi_1,\cdots,\hat{\xi_i},\cdots,\xi_n}x_j]~\rangle\\
  &=& \sum\limits_{i,j=1}^{n} (-1)^{i+j+1}\langle~^t\mu(\xi_i),x_1\otimes\cdots\otimes x_{j-1}\otimes\hat{x_j}\otimes x_{j+1}\otimes\cdots\otimes x_n\otimes\ada^\ast_{\xi_1,\cdots,\hat{\xi_i},\cdots,\xi_n}x_j~\rangle\\
  &=&\sum\limits_{i,j=1}^{n} (-1)^{i+j+1}(-1)^{n-j}\langle~^t\mu(\xi_i),x_1\otimes\cdots\otimes x_{j-1}\otimes\ada^\ast_{\xi_1,\cdots,\hat{\xi_i},\cdots,\xi_n}x_j\otimes x_{j+1}\otimes\cdots\otimes x_n~\rangle\\
  &=&\sum\limits_{i=1}^{n}(-1)^{n+i-2i}\langle~(\sum_{j=1}^n (1\otimes^{j-1} \ada_{\xi_1,\cdots,\hat{\xi_i},\cdots,\xi_n}\otimes^{n-j}1))(^t\mu(\xi_i)),x_1\otimes\cdots\otimes {x_j}\otimes\cdots\otimes x_n~\rangle.
\end{eqnarray*}
Therefore,
\begin{equation*}
  (^t\mu)([\xi_1,\cdots,\xi_n]_{\g^\ast})=\sum\limits_{i=1}^{n}(-1)^{n-i}\ada^{(n)}_{\xi_1,\cdots,\hat{\xi_i},\cdots,\xi_n}(^t\mu(\xi_i)),
\end{equation*}
which implies that $^t\mu$ is a $1$-cocycle. Then $(\g^\ast, ^t\mu)$ is an $n$-Lie bialgebra.
\end{proof}
The $n$-Lie bialgebra $(\g^\ast, ^t\mu)$ is called the {\bf dual of $n$-Lie bialgebra} $(\g,\gamma)$.

\begin{rmk}
By the Proposition \ref{pro:$n$-Lie-bialg-dual}, each $n$-Lie
bialgebra $(\g,\gamma)$ has a dual $n$-Lie bialgebra $(\g^\ast, ^t\mu)$, and the dual of $(\g^\ast, ^t\mu)$ is $(\g,\gamma)$.
\end{rmk}
\section{The double of $n$-Lie bialgebras.}\label{sec:double}

In this section, we define operad matrices of $n$-Lie bialgebras, and generalize local cocycle $3$-Lie bialgebras introduced by \cite{Bai} to local cocycle $n$-Lie bialgebras. We also establish a one-to-one correspondence between the double of $n$-Lie bialgebras and Manin triples of $n$-Lie algebras.

\subsection{Local cocycle $n$-Lie bialgebras}
In this subsection, we first introduce the notion of operad matrix, which can be utilized to represent the $1$-cocycle condition of $n$-Lie bialgebras. Then, by using the operad matrix, we define $R_i$-operad $n$-Lie bialgebras and $C_j$-operad $n$-Lie bialgebras. We also demonstrate that $R_i$-operad matrices generalize the local cocycle 3-Lie bialgebras to local cocycle $n$-Lie bialgebras. Finally, we give the relationship between $n$-Lie bialgebras and local cocycle $n$-Lie bialgebras in the Proposition \ref{lem:lo-co-$n$-Lie-bialg}.

\begin{defi}
Let $(\g,\gamma)$ be an $n$-Lie bialgebra. For all $x_1,\cdots,x_n\in \g$, we define an {\bf operad matrix} $A=(h_{ij})_{n\times n}$, where $h_{ij}\in\otimes^n End(\g)$,
$$
h_{ij}=(-1)^{n-j}1\otimes^{i-1} \ad_{\hat{x_j}}\otimes^{n-i} 1, \quad 1\leq i, j\leq n,
$$
where $\ad_{\hat{x}_j}=\ad_{x_1,\cdots,\hat{x}_j,\cdots,x_n}$, $ \ad:\wedge^{n-1}\g \rightarrow End(\g)$ is the  adjoint representation of $\g$ and $1:\g\rightarrow\g$ is an identity map, $A$ is called the {\bf operad matrix of the $n$-Lie bialgebra} $(\g,\gamma)$.
\end{defi}

\begin{pro}\label{pro:$n$-Lie-alg-ope}
Let $(\g,\gamma)$ be an $n$-Lie bialgebra and $A$ be the operad matrix of $(\g,\gamma)$. Then $\gamma$ is a $1$-cocycle if and only if
\begin{equation}\label{eq:1-cocy}
  \gamma([x_1,\cdots,x_n])=(1,1,\cdots,1)_{1\times n}~A~(\gamma(x_1), \gamma(x_2), \cdots ,\gamma(x_n) )^T.
\end{equation}
\end{pro}
\begin{proof}
 By \eqref{eq:ii'}, for all $x_1,\cdots,x_n \in \g$,
 \begin{align*}
    \gamma([x_1,\cdots,x_n])
   =&\sum\limits_{i,j=1}^{n}(-1)^{n-i}(1\otimes^{j-1}\ad_{x_1,\cdots,x_{i-1},\hat{x_{i}},x_{i+1},\cdots,x_n}\otimes^{n-j}1)\gamma(x_i)\\
   =& (1,1,\cdots,1)_{1\times n}A(\gamma(x_1), \gamma(x_2), \cdots ,\gamma(x_n))^T.
 \end{align*}
 This completes the proof.
\end{proof}

\begin{ex}
Let $(\g,\gamma)$ be a Lie bialgebra, then $\gamma$ is a $1$-cocycle if and only if for all $x_1,x_2\in\g$,
$$
\gamma([x_1,x_2])=(1,1)A(\gamma(x_1), \gamma(x_2))^T,
$$
where
$$
A=\left(
    \begin{array}{cc}
      -\ad_{x_2}\otimes1 & \ad_{x_1}\otimes1 \\
      -1\otimes\ad_{x_2} & 1\otimes\ad_{x_1} \\
    \end{array}
  \right).
$$
\end{ex}

\begin{ex}
Let $(\g,\gamma)$ be a 3-Lie bialgebra, then $\gamma$ is a $1$-cocycle if and only if for all $x_1,x_2,x_3\in\g$,
$$
\gamma([x_1,x_2,x_3])=(1,1,1)A(\gamma(x_1), \gamma(x_2), \gamma(x_3))^T,
$$
where
$$
A=\left(
    \begin{array}{ccc}
      \ad_{x_2,x_3}\otimes1\otimes1 & -\ad_{x_1,x_3}\otimes1\otimes1 & \ad_{x_1,x_2}\otimes1\otimes1 \\
      1\otimes\ad_{x_2,x_3}\otimes1 & -1\otimes\ad_{x_1,x_3}\otimes1 & 1\otimes\ad_{x_1,x_2}\otimes1 \\
      1\otimes1\otimes\ad_{x_2,x_3} & -1\otimes1\otimes\ad_{x_1,x_3} & 1\otimes1\otimes\ad_{x_1,x_2} \\
    \end{array}
  \right).
$$
\end{ex}

\begin{defi}\label{defi:$R_i$-operad}
Let $\g$ be an $n$-Lie algebra, $\gamma:\g\rightarrow\otimes^n\g$ be a linear map  such that $^t\gamma:\otimes^n\g^\ast\rightarrow\g^\ast$ defines an $n$-Lie algebra structure on $\g^\ast$.  If the operad matrix has the form
$$
A_{R_i}=\left(
          \begin{array}{cccc}
            0 & 0 & \cdots & 0 \\
            \vdots & \vdots & \ddots & \vdots \\
                        0 & 0 & \cdots & 0 \\
(-1)^{n-1}1\otimes^{i-1} \ad_{\hat{x}_1}\otimes^{n-i}1 & (-1)^{n-2}1\otimes^{i-1} \ad_{\hat{x}_2}\otimes^{n-i}1 & \cdots & 1\otimes^{i-1} \ad_{\hat{x}_n}\otimes^{n-i}1 \\
                        0 & 0 & \cdots & 0 \\
            \vdots & \vdots & \ddots & \vdots \\
            0 & 0 & \cdots & 0 \\
          \end{array}
        \right),
 $$
where $1\leq i\leq n$,  and
 \begin{equation}\label{$R$}
   \gamma([x_1,\cdots,x_n])=(1,1,\cdots,1)_{1\times n}~A_{R_i}~(\gamma(x_1), \gamma(x_2), \cdots ,\gamma(x_n) )^T,
 \end{equation}
  then  $(\g,\gamma)$ is called
  the {\bf $R_i$-operad} $n$-Lie bialgebra. For convenience, we use $\gamma_R^i$ to represent the map $\gamma$ corresponding to the operad matrix $A_{R_i}$, and denote the $R_i$-operad $n$-Lie bialgebra as $(\g,\gamma_R^i)$.
\end{defi}


\begin{ex}
Let $\g$ be a Lie algebra, and $\gamma_R^1, \gamma_R^2: \g\rightarrow\g\otimes\g$ be two linear maps such that
\begin{eqnarray*}
  \gamma_R^1([x_1,x_2]) &=& (1,1)A_{R_1}(\gamma_R^1(x_1), \gamma_R^1(x_2))^T, \\
  \gamma_R^2([x_1,x_2]) &=& (1,1)A_{R_2}(\gamma_R^2(x_1), \gamma_R^2(x_2))^T,
\end{eqnarray*}
where
$$
A_{R_1}=\left(
          \begin{array}{cc}
            -\ad_{x_2}\otimes1 & \ad_{x_1}\otimes1 \\
            0 & 0 \\
          \end{array}
        \right),
$$
and
$$
A_{R_2}=\left(
          \begin{array}{cc}
            0 & 0 \\
            -1\otimes\ad_{x_2} & 1\otimes\ad_{x_1} \\
          \end{array}
        \right).
$$
If $\gamma=\gamma_R^1+\gamma_R^2$ satisfies that $^t\gamma:\g^\ast\otimes\g^\ast\rightarrow\g^\ast$ defines a Lie algebra structure on $\g^\ast$, then $(\g,\gamma)$ is a local cocycle Lie bialgebra.
\end{ex}

\begin{rmk}\cite{Bai}
Let $(\g,\gamma)$ be a local cocycle Lie bialgebra. If the following compatibility condition holds:
\begin{equation}\label{eq:compatibility}
  (1\otimes\ad_{x_1})\gamma_1(x_2)+(\ad_{x_1}\otimes1)\gamma_2(x_2)-(1\otimes\ad_{x_2})\gamma_1(x_1)-(\ad_{x_2}\otimes1)\gamma_2(x_1)=0,
\end{equation}
then $(\g,\gamma)$ is a Lie bialgebra. Conversely, let $(\g,\gamma)$ be a Lie bialgebra. If $\gamma=\gamma_1+\gamma_2$ such that for any $x_1,x_2 \in \g$, \eqref{eq:compatibility} holds, then $(\g,\gamma)$ is a local cocycle Lie bialgebra.
\end{rmk}

Similarly, we can define $C_j$-operad $n$-Lie bialgebra.
\begin{defi}
Let $\g$ be an $n$-Lie algebra, $\gamma:\g\rightarrow\otimes^n\g$ be a linear map  such that $^t\gamma:\otimes^n\g^\ast\rightarrow\g^\ast$ defines an $n$-Lie algebra structure on $\g^\ast$. If the operad matrix has the form
$$
A_{C_j}=\left(
          \begin{array}{ccccccc}
            0 & \cdots & 0 & (-1)^{n-j} \ad_{\hat{x}_j}\otimes^{n-1}1
             & 0 & \cdots & 0 \\
            0 & \cdots & 0 & (-1)^{n-j}1\otimes \ad_{\hat{x}_j}\otimes^{n-2}1
            & 0 & \cdots & 0 \\
            \vdots & \ddots & \vdots & \vdots & \vdots
            & \ddots & \vdots \\
            0 & \cdots & 0 & (-1)^{n-j}1\otimes^{n-1} \ad_{\hat{x}_j}
            & 0 & \cdots & 0 \\
          \end{array}
        \right),
$$
where $1\leq j\leq n$,  and
\begin{equation}\label{$C$}
   \gamma([x_1,\cdots,x_n])=(1,1,\cdots,1)_{1\times n}~A_{C_j}~(\gamma(x_1), \gamma(x_2), \cdots ,\gamma(x_n) )^T,
\end{equation} then  $(\g,\gamma)$ is called
the {\bf $C_j$-operad} $n$-Lie bialgebra. We use $\gamma_C^j$ to represent the map $\gamma$ corresponding to the operad matrix $A_{C_j}$, and denote the $C_j$-operad $n$-Lie bialgebra as $(\g,\gamma_C^j)$.
\end{defi}

\begin{ex}
Let $\g$ be a Lie algebra, and $\gamma_C^1, \gamma_C^2: \g\rightarrow\g\otimes\g$ be two linear maps such that
\begin{eqnarray*}
  \gamma_C^1([x_1,x_2]) &=& (1,1)A_{C_1}(\gamma_C^1(x_1), \gamma_C^1(x_2))^T, \\
  \gamma_C^2([x_1,x_2]) &=& (1,1)A_{C_2}(\gamma_C^2(x_1), \gamma_C^2(x_2))^T,
\end{eqnarray*}
where
$$
A_{C_1}=\left(
          \begin{array}{cc}
            -\ad_{x_2}\otimes1 & 0 \\
            -1\otimes\ad_{x_2} & 0 \\
          \end{array}
        \right),
$$
and
$$
A_{C_2}=\left(
          \begin{array}{cc}
            0 & \ad_{x_1}\otimes1 \\
            0 & 1\otimes\ad_{x_1} \\
          \end{array}
        \right).
$$
If $\gamma_C^1$ (or $\gamma_C^2$) satisfies $^t\gamma_C^1$ (or $^t\gamma_C^2$) $:\g^\ast\otimes\g^\ast\rightarrow\g^\ast$ defining a Lie algebra structure on $\g^\ast$, then $(\g,\gamma_C^1)$ or $((\g,\gamma_C^2))$ is called a $C_1$ ( or $C_2$)-operad Lie bialgebra.
\end{ex}

\begin{rmk}
Let $(\g,\gamma_C^1)$ be a $C_1$-operad Lie bialgebra, and $(\g,\gamma_C^2)$ be a $C_2$-operad Lie bialgebra. If $\gamma_C^1$ and $\gamma_C^2$ satisfy
\begin{equation*}
  (\ad_{x_1}\otimes1)\gamma_1(x_2)+(1\otimes\ad_{x_1})\gamma_1(x_2)-(\ad_{x_2}\otimes1)\gamma_2(x_1)-(1\otimes\ad_{x_2})\gamma_2(x_1)=0,
\end{equation*}
then $(\g,\gamma=\gamma_C^1+\gamma_C^2)$ is a Lie bialgebra.
\end{rmk}

\begin{defi}\cite{Bai}
Let $\g$ be a $3$-Lie algebra and $\gamma_1,\gamma_2,\gamma_3:\g\rightarrow \g\otimes\g\otimes\g$ be linear maps. If $\gamma=\gamma_1+\gamma_2+\gamma_3$ satisfies that $^t\gamma:\g^\ast\otimes\g^\ast\otimes\g^\ast\rightarrow\g^\ast$ defines a $3$-Lie algebra structure on $\g^\ast$ and for any $x_1$, $x_2$, $x_3 \in \g$,
\begin{eqnarray*}
  \gamma_1([x_1,x_2,x_3]) &=& (\ad_{x_2,x_3}\otimes1\otimes1)\gamma_1(x_1)-(\ad_{x_1,x_3}\otimes1\otimes1)\gamma_1(x_2)+(\ad_{x_1,x_2}\otimes1\otimes1)\gamma_1(x_3), \\
  \gamma_2([x_1,x_2,x_3]) &=& (1\otimes\ad_{x_2,x_3}\otimes1)\gamma_2(x_1)-(1\otimes\ad_{x_1,x_3}\otimes1)\gamma_2(x_2)+(1\otimes\ad_{x_1,x_2}\otimes1)\gamma_2(x_3), \\
  \gamma_3([x_1,x_2,x_3]) &=& (1\otimes1\otimes\ad_{x_2,x_3})\gamma_3(x_1)-(1\otimes1\otimes\ad_{x_1,x_3})\gamma_3(x_2)+(1\otimes1\otimes\ad_{x_1,x_2})\gamma_3(x_3),
\end{eqnarray*}
then $(\g,\gamma)$ is a local cocycle $3$-Lie bialgebra.
\end{defi}

\begin{pro}\label{defi:local-coc}
Let $\g$ be an $n$-Lie algebra and $\gamma_1,\gamma_2,\cdots,\gamma_n:\g\rightarrow\otimes^n\g$ be linear maps such that
\begin{equation}\label{eq:gamma-12n-cc}
  \gamma_i = \gamma_{R}^i,\quad i=1,2,\cdots,n.
\end{equation}
If the linear map $\gamma=\gamma_1+\gamma_2+\cdots+\gamma_n$ satisfies that $^t\gamma:\otimes^n\g^\ast\rightarrow\g^\ast$ defines an $n$-Lie algebra structure on $\g^\ast$,
then the pair $(\g,\gamma)$ is a {\bf local cocycle $n$-Lie bialgebra}.
\end{pro}
\begin{proof}
For any $x_1,\cdots,x_n \in \g$, \eqref{eq:gamma-12n-cc} can be rewritten as,
$$
\gamma_i([x_1,\cdots,x_n])=\sum\limits_{j=1}^{n}(-1)^{n-j}\bigg(1\otimes^{i-1} \ad_{\hat{x_j}}\otimes^{n-i} 1\bigg)\gamma_i(x_j), \quad i=1,2,\cdots,n,
$$
where $\ad_{\hat{x}_j}=\ad_{x_1,\cdots,\hat{x}_j,\cdots,x_n}$, hence the conclusion holds.
\end{proof}

\begin{pro}\label{lem:lo-co-$n$-Lie-bialg}
Let $(\g,\gamma)$ be a local cocycle $n$-Lie bialgebra satisfying that
\begin{equation}
  \sum\limits_{i,j,k=1,i\neq k}^n(-1)^{n-j}(1\otimes^{k-1}\ad_{\hat{x_j}}\otimes^{n-k}1)\gamma_{i}(x_j)=0,
\end{equation}
where $\ad_{\hat{x}_j}=\ad_{x_1,\cdots,\hat{x}_j,\cdots,x_n}$, then the pair $(\g,\gamma)$ is an $n$-Lie bialgebra.
\end{pro}
\begin{proof}
By \eqref{eq:ii'} and the Proposition \ref{defi:local-coc}, we can complete the proof directly.
\end{proof}


\begin{ex}
Let $(\g,\gamma)$ be a local cocycle $3$-Lie bialgebra and for any $x_1,x_2,x_3 \in \g$,
\begin{align*}
  & \sum\limits_{j=1}^3(-1)^{3-j}(\ad_{\hat{x_j}}\otimes1\otimes1)\gamma_{2}(x_j)
  +  \sum\limits_{j=1}^3(-1)^{3-j}(\ad_{\hat{x_j}}\otimes1\otimes1)\gamma_{3}(x_j)\\
  +&\sum\limits_{j=1}^3(-1)^{3-j}(1\otimes\ad_{\hat{x_j}}\otimes1)\gamma_{1}(x_j)
  +  \sum\limits_{j=1}^3(-1)^{3-j}(1\otimes\ad_{\hat{x_j}}\otimes1)\gamma_{3}(x_j)\\
  +&\sum\limits_{j=1}^3(-1)^{3-j}(1\otimes1\otimes\ad_{\hat{x_j}})\gamma_{1}(x_j)
  +  \sum\limits_{j=1}^3(-1)^{3-j}(1\otimes1\otimes\ad_{\hat{x_j}})\gamma_{2}(x_j)\\
  =&(\ad_{x_2,x_3}\otimes1\otimes1)\gamma_{2}(x_1)-(\ad_{x_1,x_3}\otimes1\otimes1)\gamma_{2}(x_2)+(\ad_{x_1,x_2}\otimes1\otimes1)\gamma_{2}(x_3)\\
  +&(\ad_{x_2,x_3}\otimes1\otimes1)\gamma_{3}(x_1)-(\ad_{x_1,x_3}\otimes1\otimes1)\gamma_{3}(x_2)+(\ad_{x_1,x_2}\otimes1\otimes1)\gamma_{3}(x_3)\\
  +&(1\otimes\ad_{x_2,x_3}\otimes1)\gamma_{1}(x_1)-(1\otimes\ad_{x_1,x_3}\otimes1)\gamma_{1}(x_2)+(1\otimes\ad_{x_1,x_2}\otimes1)\gamma_{1}(x_3)\\
  +&(1\otimes\ad_{x_2,x_3}\otimes1)\gamma_{3}(x_1)-(1\otimes\ad_{x_1,x_3}\otimes1)\gamma_{3}(x_2)+(1\otimes\ad_{x_1,x_2}\otimes1)\gamma_{3}(x_3)\\
  +&(1\otimes1\otimes\ad_{x_2,x_3})\gamma_{1}(x_1)-(1\otimes1\otimes\ad_{x_1,x_3})\gamma_{1}(x_2)+(1\otimes1\otimes\ad_{x_1,x_2})\gamma_{1}(x_3)\\
  +&(1\otimes1\otimes\ad_{x_2,x_3})\gamma_{2}(x_1)-(1\otimes1\otimes\ad_{x_1,x_3})\gamma_{2}(x_2)+(1\otimes1\otimes\ad_{x_1,x_2})\gamma_{2}(x_3)\\
  =&0,
\end{align*}
then the pair $(\g,\gamma)$ is a $3$-Lie bialgebra.
\end{ex}

\begin{rmk}
Let $(\g,\gamma)$ be a local cocycle $n$-Lie bialgebra, and $\mu$ is the $n$-Lie bracket of $\g$, but we cannot obtain that $(\g^\ast, ^t\mu)$ is also a local cocycle $n$-Lie bialgebra, where $^t\gamma$ is the $n$-Lie bracket of $\g^\ast$. For example, the dual of a local cocycle Lie bialgebra is not a local cocycle Lie bialgebra. Therefore, the dual of a local cocycle $n$-Lie bialgebra may not be a local cocycle $n$-Lie bialgebra.
\end{rmk}

\subsection{Manin triples of $n$-Lie algebras and the double of $n$-Lie bialgebras.}

We now show that there is a one-to-one correspondence between the double of $n$-Lie bialgebras and the Manin triples of $n$-Lie algebras.

\begin{defi}\cite{Bai-R-2}
A {\bf metric $n$-Lie algebra} is a triple $\big(\g, [\cdot,\cdots,\cdot], (\cdot,\cdot)\big)$, where $(\g,[\cdot,\cdots,\cdot])$ is an $n$-Lie algebra with the $n$-Lie bracket $[\cdot,\cdots,\cdot]$,  and $(\cdot,\cdot):\g\times\g\rightarrow \mathrm{F}$ is a non-degenerate symmetric bilinear form  satisfies the invariant condition:
\begin{equation}\label{eq:$n$-Lie-inva}
  ([x_1,\cdots,x_{n-1},x_n],t)+(x_n,[x_1,\cdots,x_{n-1},t])=0,\quad\forall~ x_1,\cdots,x_n,t \in \g.
\end{equation}
\end{defi}

\begin{defi}\label{defi:Manin-tri}
A {\bf Manin triple of $n$-Lie algebras} is a triple
$((\g, [\cdot,\cdots,\cdot], (\cdot,\cdot)), \g_1, \g_2)$, where $\big(\g, [\cdot,\cdots,\cdot], (\cdot,\cdot)\big)$ is a metric $n$-Lie algebra such that
 \begin{enumerate}[(i)]
   \item $\g_1$, $\g_2$ are isotropic $n$-Lie subalgebras of $\g$, such that $\g=\g_1\oplus\g_2$ as vector spaces.
   \item For all $x_1,\cdots,x_{n-1}\in \g_1$, $y_1,\cdots,y_{n-1}\in \g_2$ and $x\in \g$, the following conditions hold:
\begin{eqnarray}
   &&(y_2,[x_1,\cdots,x_{n-1},y_1])=0,\label{eq:inv-1}\\
   &&(x_2,[y_1,\cdots,y_{n-1},x_1])=0,\label{eq:inv-2}
\end{eqnarray}
and
\begin{equation}\label{eq:inv-3}
\left\{
\begin{split}
&(x,[x_1,\cdots,x_{n-2},y_1,y_2])=0,\\
   &\quad\quad\quad\quad\quad\vdots \\
   &(x,[x_1,x_2,y_1,\cdots,y_{n-2}])=0.
\end{split}
\right.
\end{equation}
\end{enumerate}
\end{defi}
 In particular, when $n = 2$, i.e., for a metric Lie algebra, we
obtain the usual notion of Manin triple of Lie algebras which is well known in the literature \cite{Drinfeld2}. In \cite{Bai}, for the pseudo-metric 3-Lie algebras, the authors considered Manin triple of 3-Lie algebras,
which is a special case of the Definition \ref{defi:Manin-tri} when $n=3$.

\begin{defi}
Let $(\g,[\cdot,\cdots,\cdot])$ be an $n$-Lie algebra and $\gamma:\g\rightarrow\otimes^n\g$ be a linear map.  If $\gamma$ satisfies
\begin{equation}\label{eq:n-centroid}
  \gamma([x_1,x_2,\cdots,x_n]) = [\gamma(x_1),x_2,\cdots,x_n],\quad \forall~ x_1,\cdots,x_n\in\g,
\end{equation}
then we call $\gamma$ a {\bf centroid} map.
\end{defi}

\begin{defi}
Let $(\g,[\cdot,\cdots,\cdot])$ be an $n$-Lie algebra and $\gamma:\g\rightarrow\otimes^n\g$ be a linear map. For all $x_1,\cdots,x_n\in\g$, if $\gamma$ satisfies
\begin{eqnarray}
(1\otimes^{j-1}\ad_{x_2,x_3,\cdots,x_n}\otimes^{n-j}1&+&
1\otimes^{n-1}\ad_{x_2,x_3,\cdots,x_n})\gamma(x_1)=0,\label{eq:1}\\
(1\otimes^{i-1}\ad_{x_2,x_3,\cdots,x_{n-1},x_n}\otimes^{n-i}1)\gamma(x_1)
&+&(1\otimes^{k-1}\ad_{x_2,x_3,\cdots,x_{n-1},x_1}\otimes^{n-k}1)\gamma(x_n)=0,\label{eq:2}
\end{eqnarray}
where $x_i\in \g$, $1\leq j \leq n-1$, $1\leq i,k\leq n, ~i\neq k$,
then we call $\gamma$ a {\bf local operad map}. If $\gamma$ is both a centroid map and a local operad map, then we call $\gamma$ a {\bf local operad centroid} map.
\end{defi}

\begin{pro}\label{pro:double-Lie}
Let  $(\g,\gamma_{R}^1)$ be an $R_1$-operad $n$-Lie bialgebra and $\gamma_{R}^1$ be a local operad centroid map.
Set $\frkd=\g\oplus\g^\ast$, and $[\cdot,\cdots,\cdot]_\frkd:\otimes^n \frkd\rightarrow \frkd$  is a linear map, for all $x_1,\cdots,x_n\in \g, \xi_1,\cdots,\xi_n\in\g^*$ satisfying
\begin{align}\label{eq:$n$-Lie-on-d}
  [x_1+\xi_1,x_2+\xi_2,\cdots,x_n+\xi_n]_\frkd &=
  [x_1,x_2,\cdots,x_n]+\sum\limits_{i=1}^{n}(-1)^{n-i}\ad^\ast_{x_1,\cdots,x_{i-1},\hat{x_i},x_{i+1},\cdots,x_n}\xi_i \nonumber \\ &+[\xi_1,\xi_2,\cdots,\xi_n]_{\g^\ast}+\sum\limits_{i=1}^{n}(-1)^{n-i}\ada^\ast_{\xi_1,\cdots,\xi_{i-1},\hat{\xi_i},\xi_{i+1},\cdots,\xi_n}x_i,
\end{align}
where $(\g^*,\ad^*)$ is the coadjoint representation of $\g$ on $\g^*$ and $(\g,\ada^*)$ is the coadjoint representation of $\g^*$ on $\g$.
Then  $(\frkd,[\cdot,\cdots,\cdot]_\frkd)$ is an $n$-Lie algebra.
\end{pro}
\begin{proof}
It is obvious that $[\cdot,\cdots,\cdot]_\frkd$ is skew-symmetric. Now we prove that $[\cdot,\cdots,\cdot]_\frkd$ satisfies the Filippov-Jacobi identity. For all $x_1,\cdots,x_{n-1},~y_1,\cdots,y_{n} \in \g$, $\xi_1,\cdots,\xi_{n-1},~\eta_1,\cdots,\eta_{n} \in \g^{\ast}$, by \eqref{eq:$n$-Lie-on-d}, we have
\begin{eqnarray*}
  && \big[x_1+\xi_1,\cdots,x_{n-1}+\xi_{n-1},[y_1+\eta_1,\cdots,y_n+\eta_n]_{\frkd}\big]_{\frkd}  \\
  &-& \sum\limits_{i=1}^{n}\big[y_1+\eta_1,\cdots,y_{i-1}+\eta_{i-1},[x_1+\xi_1,\cdots,x_{n-1}+\xi_{n-1},y_i+\eta_i]_{\frkd},y_{i+1}+\eta_{i+1},\cdots,y_n+\eta_n\big]_{\frkd} \\
  &=& \sum_{k=1}^6\big((A_k)+(B_k)\big)-(C_1)-(C_2),
\end{eqnarray*}
where
\begin{align*}
  (A_1)&:=[x_1,\cdots,x_{n-1},[y_1,\cdots,y_n]]-\sum\limits_{i=1}^n[y_1,\cdots,y_{i-1},[x_1,\cdots,x_{n-1},y_i],y_{i+1},\cdots,y_n].\\
  (A_2)&:=\ada^{\ast}_{\xi_1,\cdots,\xi_{n-1}}\big([y_1,\cdots,y_n]\big)
  -\sum\limits_{i=1}^{n}[y_1,\cdots,y_{i-1},\ada_{\xi_1,\cdots,\xi_{n-1}}^\ast y_i,y_{i+1},\cdots,y_n].\\
  (A_3)&:=[x_1,\cdots,x_{n-1},\sum\limits_{i=1}^{n}(-1)^{n-i}\ada_{\eta_1,\cdots,\hat{\eta_i},\cdots,\eta_n}^\ast y_i]
  -\sum\limits_{i=1}^{n}(-1)^{n-i}\ada_{\eta_1,\cdots,\hat{\eta_i},\cdots,\eta_n}^\ast\big([x_1,\cdots,x_{n-1},y_i]\big) \\
  &- \sum\limits_{i=1}^{n}\sum\limits_{k=1}^{i-1}(-1)^{n-k}\ada_{\eta_1,\cdots,\hat{\eta_k},\cdots,\eta_{i-1},\ad_{x_1,\cdots,x_{n-1}}^\ast\eta_i,\eta_{i+1},\cdots,\eta_n}^\ast y_k \\
  &- \sum\limits_{i=1}^{n}\sum\limits_{k=i+1}^{n}(-1)^{n-k}\ada_{\eta_1,\cdots,\eta_{i-1},\ad_{x_1,\cdots,x_{n-1}}^\ast\eta_i,\eta_{i+1},\cdots,\hat{\eta_k},\cdots,\eta_n}^\ast y_k.\\
  (A_4) &:= \sum\limits_{j=1}^{n-1}(-1)^{n-j}\ada_{\xi_1,\cdots,\hat{\xi_j},\cdots,\xi_{n-1},\sum\limits_{i=1}^{n}(-1)^{n-i}\ad_{y_1,\cdots,\hat{y_i},\cdots,y_n}^\ast \eta_i}^\ast x_j \\
  &- \sum\limits_{i=1}^{n}[y_1,\cdots,y_{i-1},\sum\limits_{j=1}^{n-1}(-1)^{n-j}\ada_{\xi_1,\cdots,\hat{\xi_j},\cdots,\xi_{n-1},\eta_i}^\ast x_j,y_{i+1},\cdots,y_n].\\
 (A_5) &:= \ada^{\ast}_{\xi_1,\cdots,\xi_{n-1}}\big(\sum\limits_{i=1}^{n}(-1)^{n-i}\ada_{\eta_1,\cdots,\hat{\eta_i},\cdots,\eta_n}^\ast y_i\big)
 -\sum\limits_{i=1}^{n}(-1)^{n-i}\ada_{\eta_1,\cdots,\hat{\eta_i},\cdots,\eta_n}^\ast\big(\ada_{\xi_1,\cdots,\xi_{n-1}}^\ast y_i\big) \\
 &- \sum\limits_{i=1}^{n}\sum\limits_{k=1}^{i-1}(-1)^{n-k}\ada_{\eta_1,\cdots,\hat{\eta_k},\cdots,\eta_{i-1},[\xi_1,\cdots,\xi_{n-1},\eta_i]_{\g^\ast},\eta_{i+1},\cdots,\eta_n}^\ast y_k\\
 &-\sum\limits_{i=1}^{n}\sum\limits_{k=i+1}^{n}(-1)^{n-k}\ada_{\eta_1,\cdots,\eta_{i-1},[\xi_1,\cdots,\xi_{n-1},\eta_i]_{\g^\ast},\eta_{i+1},\cdots,\hat{\eta_k},\cdots,\eta_n}^\ast y_k.\\
(A_6)&:=\sum\limits_{j=1}^{n-1}(-1)^{n-j}\ada_{\xi_1,\cdots,\hat{\xi_j},\cdots,\xi_{n-1},[\eta_1,\cdots,\eta_n]_{\g^\ast}}^\ast x_j-\sum\limits_{i=1}^{n}(-1)^{n-i}\ada_{\eta_1,\cdots,\hat{\eta_i},\cdots,\eta_n}^\ast\big(\sum\limits_{j=1}^{n-1}(-1)^{n-j}\ada_{\xi_1,\cdots,\hat{\xi_j},\cdots,\xi_{n-1},\eta_i}^\ast x_j\big).\\
  (B_1)&:=\big[\xi_1,\cdots,\xi_{n-1},[\eta_1,\cdots,\eta_n]_{\g^\ast}\big]_{\g^\ast}-\sum\limits_{i=1}^{n}\big[\eta_1,\cdots,\eta_{i-1},[\xi_1,\cdots,\xi_{n-1},\eta_i]_{\g^\ast},\eta_{i+1},\cdots,\eta_n\big]_{\g^\ast}.\\
  (B_2)&:=\ad_{x_1,\cdots,x_{n-1}}\big([\eta_1,\cdots,\eta_n]_{\g^\ast}\big)-\sum\limits_{i=1}^{n}[\eta_1,\cdots,\eta_{i-1},\ad_{x_1,\cdots,x_{n-1}}^\ast\eta_i,\eta_{i+1},\cdots,\eta_n]_{\g^\ast}.\\
  (B_3) &:= [\xi_1,\cdots,\xi_{n-1},\sum\limits_{i=1}^{n}(-1)^{n-i}\ad_{y_1,\cdots,\hat{y_i},\cdots,y_n}^\ast \eta_i]_{\g^\ast}-\sum\limits_{i=1}^{n}(-1)^{n-i}\ad_{y_1,\cdots,\hat{y_i},\cdots,y_n}^\ast\big([\xi_1,\cdots,\xi_{n-1},\eta_i]_{\g^\ast}\big) \\
  &- \sum\limits_{i=1}^{n}\sum\limits_{k=1}^{i-1}(-1)^{n-k}\ad_{y_1,\cdots,\hat{y_k},\cdots,y_{i-1},\ada_{\xi_1,\cdots,\xi_{n-1}}^\ast y_i,y_{i+1},\cdots,y_n}^\ast \eta_k \\
  &- \sum\limits_{i=1}^{n}\sum\limits_{k=i+1}^{n}(-1)^{n-k}\ad_{y_1,\cdots,y_{i-1},\ada_{\xi_1,\cdots,\xi_{n-1}}^\ast y_i,y_{i+1},\cdots,\hat{y_k},\cdots,y_n}^\ast\eta_k.\\
  (B_4) &:= \sum\limits_{j=1}^{n-1}(-1)^{n-j}\ad_{x_1,\cdots,\hat{x_j},\cdots,x_{n-1},\sum\limits_{i=1}^{n}(-1)^{n-i}\ada_{\eta_1,\cdots,\hat{\eta_i},\cdots,\eta_n}^\ast y_i}^\ast \xi_j \\
  &- \sum\limits_{i=1}^{n}[\eta_1,\cdots,\eta_{i-1},\sum\limits_{j=1}^{n-1}(-1)^{n-j}\ad_{x_1,\cdots,\hat{x_j},\cdots,x_{n-1},y_i}^\ast \xi_j,\eta_{i+1},\cdots,\eta_n]_{\g^\ast}.\\
(B_5) &:= \ad^{\ast}_{x_1,\cdots,x_{n-1}}\big(\sum\limits_{i=1}^{n}(-1)^{n-i}\ad_{y_1,\cdots,\hat{y_i},\cdots,y_n}^\ast \eta_i\big)-\sum\limits_{i=1}^{n}(-1)^{n-i}\ad_{y_1,\cdots,\hat{y_i},\cdots,y_n}^\ast\big(\ad_{x_1,\cdots,x_{n-1}}^\ast \eta_i\big)\\
&- \sum\limits_{i=1}^{n}\sum\limits_{k=1}^{i-1}(-1)^{n-k}\ad_{y_1,\cdots,\hat{y_k},\cdots,y_{i-1},[x_1,\cdots,x_{n-1},y_i],y_{i+1},\cdots,y_n}^\ast \eta_k\\
&-\sum\limits_{i=1}^{n}\sum\limits_{k=i+1}^{n}(-1)^{n-k}\ad_{y_1,\cdots,y_{i-1},[x_1,\cdots,x_{n-1},y_i],y_{i+1},\cdots,\hat{y_k},\cdots,y_n}^\ast\eta_k.\\
(B_6)&:=\sum\limits_{j=1}^{n-1}(-1)^{n-j}\ad_{x_1,\cdots,\hat{x_j},\cdots,x_{n-1},[y_1,\cdots,y_n]}^\ast \xi_j-\sum\limits_{i=1}^{n}(-1)^{n-i}\ad_{y_1,\cdots,\hat{y_i},\cdots,y_n}^\ast\big(\sum\limits_{j=1}^{n-1}(-1)^{n-j}\ad_{x_1,\cdots,\hat{x_j},\cdots,x_{n-1},y_i}^\ast \xi_j\big).\\
  (C_1) &:= \sum\limits_{i=1}^{n}\sum\limits_{k=1}^{i-1}(-1)^{n-k}\ada_{\eta_1,\cdots,\hat{\eta_k},\cdots,\eta_{i-1},\sum\limits_{j=1}^{n-1}(-1)^{n-j}\ad_{x_1,\cdots,\hat{x_j},\cdots,x_{n-1},y_i}^\ast\xi_j,\eta_{i+1},\cdots,\eta_n}^\ast y_k \\
  &+ \sum\limits_{i=1}^{n}\sum\limits_{k=i+1}^{n}(-1)^{n-k}\ada_{\eta_1,\cdots,\eta_{i-1},\sum\limits_{j=1}^{n-1}(-1)^{n-j}\ad_{x_1,\cdots,\hat{x_j},\cdots,x_{n-1},y_i}^\ast\xi_j,\eta_{i+1},\cdots,\hat{\eta_k},\cdots,\eta_n}^\ast y_k.\\
(C_2) &:= \sum\limits_{i=1}^{n}\sum\limits_{k=1}^{i-1}(-1)^{n-k}\ad_{y_1,\cdots,\hat{y_k},\cdots,y_{i-1},\sum\limits_{j=1}^{n-1}(-1)^{n-j}\ada_{\xi_1,\cdots,\hat{\xi_j},\cdots,\xi_{n-1},\eta_i}^\ast x_j,y_{i+1},\cdots,y_n}^\ast \eta_k\\
&+\sum\limits_{i=1}^{n}\sum\limits_{k=i+1}^{n}(-1)^{n-k}\ad_{y_1,\cdots,y_{i-1},\sum\limits_{j=1}^{n-1}(-1)^{n-j}\ada_{\xi_1,\cdots,\hat{\xi_j},\cdots,\xi_{n-1},\eta_i}^\ast x_j,y_{i+1},\cdots,\hat{y_k},\cdots,y_n}^\ast\eta_k.
\end{align*}

In the following, we prove that $(A_i),(B_j),(C_k)$ are zero, respectively, $1\leq i,j\leq6, 1\leq k\leq2$.

$\bf(1).$ By the Filippov-Jacobi identity on $\g$ and $\g^{\ast}$, it is obvious that $(A_1)$ and $(B_1)$ are zero. Let $i=1$ in the Definition \ref{defi:$R_i$-operad}, 
we obtain
\begin{equation}
 \gamma_\g([x_1,\cdots,x_n]) = \sum\limits_{i=1}^{n}(-1)^{n-i}(\ad_{x_1,\cdots,x_{i-1},
\hat{x_i},x_{i+1},\cdots,x_n}\otimes^{n-1}1)\gamma_\g(x_i).
\end{equation}
For any $\xi_1\otimes\cdots\otimes\xi_n\in \otimes^n\g^*$, we have
\begin{align*}
 \langle \gamma_\g([x_1,\cdots,x_n]), \xi_1\otimes\cdots\otimes\xi_n\rangle
&=-\langle  \ada_{\xi_1,\cdots,\xi_{n-1}}^\ast([x_1,\cdots,x_n]), \xi_n\rangle,
\end{align*}
and
\begin{eqnarray*}
&&\sum\limits_{i=1}^{n}(-1)^{n-i}\langle(\ad_{x_1,\cdots,x_{i-1},
\hat{x_i},x_{i+1},\cdots,x_n}\otimes^{n-1}1)\gamma_\g(x_i),
\xi_1\otimes\cdots\otimes\xi_n\rangle\\
&=&-\sum\limits_{i=1}^{n}\langle[x_1,\cdots,x_{i-1},\ada_{\xi_1,\cdots,\xi_{n-1}}^\ast x_i,x_{i+1},\cdots,x_n], \xi_n\rangle.
\end{eqnarray*}
Then $(A_2)$ is zero.

$\bf(2).$ By \eqref{eq:n-centroid}, for all $y_1,\cdots,y_n\in\g$ we obtain
\begin{equation}\label{eq:n-centroid-2}
 \gamma_\g([x_1,\cdots,x_{n-1},\sum_{i=1}^ny_i]) = \sum_{i=1}^n\sum\limits_{k=1}^{n}(1\otimes^{k-1}\ad_{x_1,
\cdots,x_{n-1}}\otimes^{n-k}1)\gamma_\g(y_i).
\end{equation}
Then for any $\eta_1\otimes\cdots\otimes\eta_n\in \otimes^n\g^*$,
\begin{eqnarray*}
&&\langle \gamma_\g([x_1,\cdots,x_{n-1},\sum_{i=1}^ny_i]), \eta_1\otimes\cdots\otimes\eta_n\rangle\\
&=&-\sum_{i=1}^n(-1)^{n-i}\langle  \ada_{\eta_1,\cdots,\eta_{i-1},\hat{\eta}_i,\eta_{i+1},\cdots\eta_n}^\ast([x_1,\cdots,x_{n-1},y_i]), \eta_i\rangle,
\end{eqnarray*}

\begin{eqnarray*}
&&\sum\limits_{i=1}^{n}\sum_{k=1}^{n}\langle (1\otimes^{k-1}\ad_{x_1,
\cdots,x_{n-1}}\otimes^{n-k}1)\gamma_\g(y_i),\eta_1\otimes\cdots\otimes\eta_n\rangle \\
&=&\sum\limits_{i=1}^{n}\sum_{k=1}^{i-1}(-1)^{(n-k)}\langle
\ada^*_{\eta_1,\cdots,\eta_{k-1},\hat{\eta}_k,\eta_{k+1},\cdots,
 \ad_{x_1,\cdots,x_{n-1}}^*(\eta_i),\cdots,\eta_n}(y_k),
\eta_k\rangle \\
&+&\sum_{i=1}^{n}\sum_{k=i+1}^{n}(-1)^{(n-k)}\langle
\ada^*_{\eta_1,\cdots,\ad_{x_1,\cdots,x_{n-1}}^*(\eta_i), \cdots,\eta_{k-1},\hat{\eta}_k,\eta_{k+1},\cdots\eta_n}
 (y_k),
\eta_k\rangle\\
&-&\sum_{i=1}^{n}(-1)^{(n-i)}\langle [x_1,\cdots,x_{n-1},\ada_{\eta_1,\cdots,\eta_{i-1},\hat{\eta}_{i},
  \eta_{i+1},\cdots,\eta_n}^* (y_i)],\eta_i\rangle. \\
\end{eqnarray*}
It is proved that $(A_3)$ is zero.

$\bf(3).$ By \eqref{eq:1}, for the integer $1\leq i \leq n$, $y_i, x_1, \cdots, x_{n-1} \in \g$, we can get the following equations:
\begin{equation*}
  (1\otimes^{j-1}\ad_{y_1,\cdots,y_{i-1},\hat{y_i},y_{i+1},\cdots,y_n}\otimes^{n-j}1
  +1\otimes^{n-1}\ad_{y_1,\cdots,y_{i-1},\hat{y_i},y_{i+1},\cdots,y_n})\gamma(x_j)=0,
  \quad\forall j=1,2,\cdots,n-1.
\end{equation*}
Then
\begin{equation}\label{eq:3}
  \sum\limits_{i=1}^{n}\sum\limits_{j=1}^{n-1}(-1)^{n-i}\big((1\otimes^{j-1}\ad_{y_1,\cdots,\hat{y_i},\cdots,y_n}\otimes^{n-j}1) +(1\otimes^{n-1}\ad_{y_1,\cdots,\hat{y_i},\cdots,y_n})\big)\gamma(x_j)
  =0.
\end{equation}
For any $\xi_1\otimes\cdots\otimes\xi_{n-1}\otimes\eta_i \in \otimes^n\g^\ast$, by \eqref{eq:3},
\begin{eqnarray*}
&&\langle~\sum\limits_{i=1}^{n}\sum\limits_{j=1}^{n-1}(-1)^{n-i}\big((1\otimes^{j-1}\ad_{y_1,\cdots,\hat{y_i},\cdots,y_n}\otimes^{n-j}1)\\
&+& (1\otimes^{n-1}\ad_{y_1,\cdots,\hat{y_i},\cdots,y_n})\big)\gamma(x_j),~ \xi_1\otimes\cdots\otimes\xi_{n-1}\otimes\eta_i~\rangle\\
&=&-\sum\limits_{i=1}^{n}\sum\limits_{j=1}^{n-1}(-1)^{n-j}\langle~[y_1,\cdots,y_{i-1}, \ada^\ast_{\xi,\cdots,\xi_{j-1},\hat{\xi_j},\xi_{j+1},\cdots,\xi_{n-1},\eta_i}x_j,y_{i+1},\cdots,y_n], \xi_j~\rangle\\
&+&\sum\limits_{i=1}^{n}\sum\limits_{j=1}^{n-1}(-1)^{n-i}(-1)^{n-j}\langle~\ada^\ast_{\xi,\cdots,\xi_{j-1},\hat{\xi_j},\xi_{j+1},\cdots,\xi_{n-1},\ad_{y_1,\cdots,y_{i-1},\hat{y_i},y_{i+1},\cdots,y_n}^\ast\eta_i}x_j, \xi_j~\rangle).
\end{eqnarray*}
Therefore, $(A_4)$ is zero.

$\bf (4).$ By \eqref{eq:2}, for $1\leq i, j, k, l\leq n$, $i\neq k$, $x_1,\cdots,x_{n-1},y_1,\cdots,y_n\in \g$, we have
\begin{equation*}
  (1\otimes^{i-1}\ad_{x_1,\cdots,\hat{x_j},\cdots,x_{n-1},y_i}\otimes^{n-i}1)\gamma(y_l)
  +(1\otimes^{k-1}\ad_{x_1,\cdots,\hat{x_j},\cdots,x_{n-1},y_l}\otimes^{n-k}1)\gamma(y_i) = 0.
\end{equation*}
Then
\begin{equation*}
  \sum\limits_{j=1}^{n-1}\sum\limits_{k=1,k\neq i}^{n}\big((1\otimes^{i-1}\ad_{x_1,\cdots,\hat{x_j},\cdots,x_{n-1},y_i}\otimes^{n-i}1)\gamma(y_1)
  +(1\otimes^{k-1}\ad_{x_1,\cdots,\hat{x_j},\cdots,x_{n-1},y_1}\otimes^{n-k}1)\gamma(y_i)\big)=0,
\end{equation*}
and
\begin{eqnarray}\label{eq:4}
  &&\sum\limits_{i=1}^{n}\sum\limits_{j=1}^{n-1}\sum\limits_{k=1}^{n}(-1)^{n-j}(1\otimes^{i-1}\ad_{x_1,\cdots,x_{j-1},\hat{x_j},x_{j+1},\cdots,x_{n-1},y_i}\otimes^{n-i}1)\gamma(y_k)  \\
  \nonumber&=&\sum\limits_{i=1}^{n}\sum\limits_{j=1}^{n-1}(-1)^{n-j}(1\otimes^{i-1}\ad_{x_1,\cdots,x_{j-1},\hat{x_j},x_{j+1},\cdots,x_{n-1},y_i}\otimes^{n-i}1)\gamma(y_i).
\end{eqnarray}
For any $\eta_1\otimes\cdots\otimes\eta_{i-1}\otimes\xi_j\otimes\eta_{i+1}\otimes\cdots\otimes\eta_n \in \otimes^n\g^\ast$, by \eqref{eq:4}
\begin{eqnarray*}
  &&\sum\limits_{i=1}^{n}\sum\limits_{j=1}^{n-1}\sum\limits_{k=1}^{n}(-1)^{n-j}\langle~(1\otimes^{i-1}\ad_{x_1,\cdots,x_{j-1},\hat{x_j},x_{j+1},\cdots,x_{n-1},y_i}\otimes^{n-i}1)\gamma(y_k),\\
  &&\eta_1\otimes\cdots\otimes\eta_{i-1}\otimes\xi_j\otimes\eta_{i+1}\otimes\cdots\otimes\eta_n~\rangle\\
  &=&\sum\limits_{i=1}^{n}\sum\limits_{j=1}^{n-1}\sum\limits_{k=1}^{i-1}(-1)^{n-j}(-1)^{n-k}\langle~\ada^\ast_{\eta_1,\cdots,\hat{\eta_k},\cdots,\eta_{i-1},\ad_{x_1,\cdots,x_{j-1},\hat{x_j},x_{j+1},\cdots,x_{n-1},y_i}\xi_j,\eta_{i+1},\cdots,\eta_n}y_k,\eta_k~\rangle \\
  &&+\sum\limits_{i=1}^{n}\sum\limits_{j=1}^{n-1}\sum\limits_{k=i+1}^{n}(-1)^{n-j}(-1)^{n-k}\langle~\ada^\ast_{\eta_1,\cdots,\eta_{i-1},\ad_{x_1,\cdots,x_{j-1},\hat{x_j},x_{j+1},\cdots,x_{n-1},y_i}\xi_j,\eta_{i+1},\cdots,\hat{\eta_k},\cdots,\eta_n}y_k,\eta_k~\rangle
  \\
  &&-\sum\limits_{i=1}^{n}\sum\limits_{j=1}^{n-1}(-1)^{n-j}\langle~y_i,[\eta_1,\cdots,\eta_{i-1},\ad_{x_1,\cdots,x_{j-1},\hat{x_j},x_{j+1},\cdots,x_{n-1},y_i}\xi_j,\eta_{i+1},\cdots,\eta_n]_{\g^\ast}~\rangle,
\end{eqnarray*}
and
\begin{align*}
&\sum\limits_{i=1}^{n}\sum\limits_{j=1}^{n-1}(-1)^{n-j}\langle~(1\otimes^{i-1}\ad_{x_1,\cdots,x_{j-1},\hat{x_j},x_{j+1},\cdots,x_{n-1},y_i}\otimes^{n-i}1)\gamma(y_i),\\
&\quad\quad\quad\quad\quad\quad\eta_1\otimes\cdots\otimes\eta_{i-1}\otimes\xi_j\otimes\eta_{i+1}\otimes\cdots\otimes\eta_n~\rangle \\
=&-\sum\limits_{i=1}^{n}\sum\limits_{j=1}^{n-1}(-1)^{n-j}\langle~y_i,[\eta_1,\cdots,\eta_{i-1},\ad_{x_1,\cdots,x_{j-1},\hat{x_j},x_{j+1},\cdots,x_{n-1},y_i}\xi_j,\eta_{i+1},\cdots,\eta_n]_{\g^\ast}~\rangle.
\end{align*}
Then $(C_1)$ is zero.

Similarly, $(B_2),(B_3),(B_4),(C_2)$ are zero.

$\bf(5).$ By \eqref{eq:FJi}, for all $y_1,\cdots,y_n \in \g,\xi_1,\cdots,\xi_{n-1},\eta_1,\cdots,\eta_n \in \g^\ast$,
\begin{align*}\label{9}
  & \langle~\sum\limits_{i=1}^{n}y_i,[\xi_1,\cdots,\xi_{n-1},[\eta_1,\cdots,\eta_n]_{\g^\ast}]_{\g^\ast}~\rangle \\
  \nonumber=&\langle~\sum\limits_{i=1}^{n}y_i,\sum\limits_{k=1}^{n}[\eta_1,\cdots,\eta_{k-1},[\xi_1,\cdots,\xi_{n-1},\eta_k]_{\g^\ast},\eta_{k+1},\cdots,\eta_n]_{\g^\ast}~\rangle.
\end{align*}
Then we have
\begin{align*}
  &\langle~\sum\limits_{i=1}^{n}y_i,[\xi_1,\cdots,\xi_{n-1},[\eta_1,\cdots,\eta_n]_{\g^\ast}]_{\g^\ast}~\rangle \\
  =& \langle~\sum\limits_{i=1}^{n}(-1)^{n-i}\ada^\ast_{\eta_1,\cdots,\eta_{i-1},\hat\eta_i,\eta_{i+1},\cdots,\eta_n}(\ada^\ast_{\xi_1,\cdots,\xi_{n-1}}y_i),\eta_i~\rangle,
\end{align*}
and
\begin{align*}
  & \langle~\sum\limits_{i=1}^{n}y_i,\sum\limits_{k=1}^{n}[\eta_1,\cdots,\eta_{k-1},[\xi_1,\cdots,\xi_{n-1},\eta_k]_{\g^\ast},\eta_{k+1},\cdots,\eta_n]_{\g^\ast}~\rangle \\
  =& \langle~\ada^\ast_{\xi_1,\cdots,\xi_{n-1}}(\sum\limits_{i=1}^{n}(-1)^{n-i}\ada^\ast_{\eta_1,\cdots,\hat\eta_i,\cdots,\eta_n}y_i),\eta_i~\rangle \\
  -&\langle~\sum\limits_{i=1}^{n}\sum\limits_{k=1}^{i-1}(-1)^{n-k}\ada^\ast_{\eta_1,\cdots,\hat\eta_k,\cdots,\eta_{i-1},[\xi_1,\cdots,\xi_{n-1},\eta_i]_{\g^\ast},\eta_{i+1},\cdots,\eta_n}y_k,\eta_k~\rangle
  \\
  -&\langle~\sum\limits_{i=1}^{n}\sum\limits_{k=i+1}^{n}(-1)^{n-k}\ada^\ast_{\eta_1,\cdots,\eta_{i-1},[\xi_1,\cdots,\xi_{n-1},\eta_i]_{\g^\ast},\eta_{i+1},\cdots,\hat\eta_k,\cdots,\eta_n}y_k,\eta_k~\rangle.
\end{align*}
Then $(A_5)$ is zero.

$\bf(6).$ Thanks to \eqref{eq:FJi},
\begin{align*}
  &[\xi_1,\cdots,\xi_{n-1},[\eta_1,\cdots,\eta_n]_{\g^\ast}]_{\g^\ast} \\
  \nonumber=&\sum\limits_{i=1}^{n}\sum\limits_{j=1}^{n-1}(-1)^{n-i}(-1)^{n-j}[\xi_1,\cdots,\hat{\xi_j},\cdots,\xi_{n-1},\eta_i,[\eta_1,\cdots,\eta_{i-1},\hat{\eta_i},\eta_{i+1},\cdots,\eta_n,\xi_j]_{\g^\ast}]_{\g^\ast} \\
  \nonumber&+ \sum\limits_{i=1}^{n}[\xi_1,\cdots,\xi_{n-1},[\eta_1,\cdots,\eta_n]_{\g^\ast}]_{\g^\ast},
\end{align*}
thus, we obtain
\begin{eqnarray*}\label{eq:FJi-biyao2}
&&-(n-1)[\xi_1,\cdots,\xi_{n-1},[\eta_1,\cdots,\eta_n]_{\g^\ast}]_{\g^\ast}\\
\nonumber&&=\sum\limits_{i=1}^{n}\sum\limits_{j=1}^{n-1}(-1)^{n-i}(-1)^{n-j}[\xi_1,\cdots,\hat{\xi_j},\cdots,\xi_{n-1},\eta_i,[\eta_1,\cdots,\eta_{i-1},\hat{\eta_i},\eta_{i+1},\cdots,\eta_n,\xi_j]_{\g^\ast}]_{\g^\ast},
\end{eqnarray*}
and
\begin{eqnarray*}\label{eq:FJi-biyao22}
&&-\sum\limits_{j=1}^{n-1}\langle~[\xi_1,\cdots,\xi_{n-1},[\eta_1,\cdots,\eta_n]_{\g^\ast}]_{\g^\ast},x_j~\rangle \\
\nonumber&&=\sum\limits_{i=1}^{n}\sum\limits_{j=1}^{n-1}(-1)^{n-i}(-1)^{n-j}\langle~[\xi_1,\cdots,\hat{\xi_j},\cdots,\xi_{n-1},\eta_i,[\eta_1,\cdots,\eta_{i-1},\hat{\eta_i},\eta_{i+1},\cdots,\eta_n,\xi_j]_{\g^\ast}]_{\g^\ast},x_j~\rangle.
\end{eqnarray*}
Since
\begin{eqnarray*}
  &&-\sum\limits_{j=1}^{n-1}\langle~[\xi_1,\cdots,\xi_{n-1},[\eta_1,\cdots,\eta_n]_{\g^\ast}]_{\g^\ast},x_j~\rangle
  = \sum\limits_{j=1}^{n-1}(-1)^{n-j}\langle~\xi_j,\ada^\ast_{\xi_1,\cdots,\xi_{j-1},\hat{\xi_j},\xi_{j+1},\cdots,\xi_{n-1},[\eta_1,\cdots,\eta_n]_{\g^\ast}}x_j~\rangle,
\end{eqnarray*}
and
\begin{eqnarray*}
  && \sum\limits_{i=1}^{n}\sum\limits_{j=1}^{n-1}(-1)^{n-i}(-1)^{n-j}\langle~[\xi_1,\cdots,\hat{\xi_j},\cdots,\xi_{n-1},\eta_i,[\eta_1,\cdots,\eta_{i-1},\hat{\eta_i},\eta_{i+1},\cdots,\eta_n,\xi_j]_{\g^\ast}]_{\g^\ast},x_j~\rangle\\
  &=& \sum\limits_{i=1}^{n}\sum\limits_{j=1}^{n-1}(-1)^{n-i}(-1)^{n-j}\langle~\xi_j,\ada^\ast_{\eta_1,\cdots,\eta_{i-1},\hat{\eta_i},\eta_{i+1},\cdots,\eta_n}(\ada^\ast_{\xi_1,\cdots,\xi_{j-1},\hat{\xi_j},\xi_{j+1},\cdots,\xi_{n-1},\eta_i}x_j)~\rangle,
\end{eqnarray*}
we get $(A_6)$ is zero.

By the completely similar discussion to the above, we get $(B_5),(B_6)$ are zero.
Hence $(\frkd,[\cdot,\cdots,\cdot]_\frkd)$ is an $n$-Lie algebra.
\end{proof}
\begin{rmk}
When $n = 2$, it is exactly the conclusion for Lie bialgebras. When $n=3$,
 for all $~x_1,x_2,x_3\in \g, \xi_1,\xi_2,\xi_3\in\g^*$, the linear map $[\cdot,\cdot,\cdot]_\frkd:\wedge^3 \frkd\rightarrow \frkd$ satisfying
\begin{eqnarray*}
  [x_1+\xi_1,x_2+\xi_2,x_3+\xi_3]_\frkd &=&
  [x_1,x_2,x_3]+\ad^\ast_{x_2,x_{3}}\xi_1 -\ad^\ast_{x_1,x_{3}}\xi_2+\ad^\ast_{x_1,x_{2}}\xi_3\\ &+&[\xi_1,\xi_2,\xi_3]_{\g^\ast}+\ada^\ast_{\xi_2,\xi_{3}}x_1
-\ada^\ast_{\xi_1,\xi_{3}}x_2+\ada^\ast_{\xi_1,\xi_{2}}x_3,
\end{eqnarray*}
is a 3-Lie bracket, see \cite{Bai} for more details.
\end{rmk}

\begin{defi}\label{defi:doub-Li-bialg}

Let  $(\g,\gamma_{R}^1)$ be an $R_1$-operad $n$-Lie bialgebra. If $\gamma_{R}^1$ is a local operad centroid map, then we call $(\g,\gamma_{R}^1)$  a \bf{double construction $n$-Lie bialgebra}.
\end{defi}

\begin{pro}
Let  $(\g,\gamma)$ be an $n$-dimensional  double construction $n$-Lie bialgebra with a basis $\{e_1,\cdots,e_n\}$, and
\begin{equation*}
  [e_{a_1},\cdots,e_{a_n}]=\sum\limits_{k=1}^nT_{a_1,\cdots,a_n}^k e_k,
  \quad \gamma(e_{i})=\sum\limits_{s_1,\cdots,s_n=1}^nC_{i}^{s_1,\cdots,s_n}e_{s_1}\otimes\cdots\otimes e_{s_n},
\end{equation*}
where $1 \leq a_1,\cdots,a_n, s_1,\cdots,s_n, i,k\leq n$ are positive integers, $T_{a_1,\cdots,a_n}^k $, $C_{i}^{s_1,\cdots,s_n}\in F$. Then
$\gamma$ satisfies \eqref{eq:n-centroid}, \eqref{eq:1} and \eqref{eq:2} if and only if the following \eqref{eq:iff30}, \eqref{eq:iff31} and \eqref{eq:iff32} hold, respectively,
\begin{equation}\label{eq:iff30}
  \sum\limits_{s_1,\cdots,s_n,k=1}^n\big(T_{a_1,\cdots,a_n}^k C_k^{s_1,\cdots,s_n}-\sum\limits_{i=1}^n(-1)^{n-1}T_{a_2,\cdots,a_n,k}^{s_i} C_{a_1}^{s_1,\cdots,s_{i-1},\hat{s_i},s_{i+1},\cdots,s_n,k}\big)=0,
\end{equation}
\begin{equation}\label{eq:iff31}
  \sum\limits_{s_1,\cdots,s_n,k=1}^n T_{a_2,\cdots,a_n,s_j}^k C_{a_1}^{s_1,\cdots,s_n}=\sum\limits_{s_1,\cdots,s_n,k=1}^n T_{a_2,\cdots,a_n,s_n}^{k} C_{a_1}^{s_1,\cdots,s_n}=0,\quad \forall~ j=1,2,\cdots,n-1,
\end{equation}
\begin{equation}\label{eq:iff32}
  \sum\limits_{s_1,\cdots,s_n,j=1}^n T_{a_2,\cdots,a_n,s_i}^j C_{a_1}^{s_1,\cdots,s_n}=\sum\limits_{s_1,\cdots,s_n,j=1}^n T_{a_1,\cdots,a_{n-1},s_k}^j C_{a_n}^{s_1,\cdots,s_n}=0,\quad \forall~ i,k=1,2,\cdots,n;~i\neq k.
\end{equation}
\end{pro}

\begin{proof}
The result follows from the direct computation according to \eqref{eq:n-centroid}, \eqref{eq:1} and \eqref{eq:2}.
\end{proof}


%
%
%

\begin{thm}\label{main}
Let  $\g$ be an $n$-Lie algebra, $\g^*$ be the dual space of $\g$, and $\gamma_\g:\g\rightarrow \otimes^n\g$ be a linear map such that
$^t\gamma_\g: \otimes^n\g^*\rightarrow \g^*$ defines an $n$-Lie algebra structure on $\g^*$.
Then $(\g,\gamma_{\g})$ is a double construction $n$-Lie bialgebra if and only if $((\frkd,[\cdot,\cdots,\cdot]_\frkd,\langle\cdot,\cdot\rangle_\frkd),\g,\g^\ast)$ is a Manin triple of $n$-Lie algebra, where $\frkd=\g\oplus \g^*$, $[\cdot,\cdots,\cdot]_\frkd:\otimes^n \frkd \rightarrow \frkd$ is given by \eqref{eq:$n$-Lie-on-d}, and
$\langle\cdot,\cdot\rangle_\frkd:\frkd\times \frkd\rightarrow F$ is a bilinear form given by
\begin{equation}\label{eq:d-bili}
  \langle~x+\xi,y+\eta~\rangle_\frkd =\langle~\eta,x~\rangle+\langle~\xi,y~\rangle, \quad\forall~ x,y \in \g, \xi,\eta \in \g^\ast.
\end{equation}

\end{thm}

\begin{proof}


It is clear that $\langle\cdot,\cdot\rangle_\frkd$ defined by \eqref{eq:d-bili} is a non-degenerate symmetric bilinear form on $\frkd$, $\g$ and $\g^*$ are isotropic with respect to $\langle\cdot,\cdot\rangle_\frkd$. Thanks to the Proposition \ref{pro:double-Lie}, if $(\g,\gamma_{\g})$ is a double construction $n$-Lie bialgebra, then $[\cdot,\cdots,\cdot]_\frkd$ is an $n$-Lie bracket that satisfies \eqref{eq:inv-1}, \eqref{eq:inv-2} and \eqref{eq:inv-3}.
Therefore, 
$((\frkd,[\cdot,\cdots,\cdot]_\frkd,\langle\cdot,\cdot\rangle_\frkd),\g,\g^\ast)$ is a Manin triple of $n$-Lie algebra.

Conversely, suppose that $((\frkd,[\cdot,\cdots,\cdot]_\frkd,(\cdot,\cdot)),\g,\g^\ast)$ is a Manin triple of $n$-Lie algebra. By \eqref{eq:$n$-Lie-inva},
for all  $x_1,\cdots,x_n\in \g, \xi\in\g^*$, we have
\begin{align*}
  &\langle x_n,[x_1,\cdots,x_{n-1},\xi]_\frkd\rangle_{\frkd}
  = -\langle [x_1,\cdots,x_{n-1},x_n]_\frkd,\xi\rangle_{\frkd}
  = -\langle[x_1,\cdots,x_{n-1},x_n],\xi\rangle_{\frkd}.
\end{align*}
By \eqref{eq:d-bili},
\begin{align*}
  -\langle~\xi,[x_1,\cdots,x_{n-1},x_n]~\rangle
  = \langle~\ad^\ast_{x_1,\cdots,x_{n-1}}\xi,x_n~\rangle
  = \langle x_n,\ad^\ast_{x_1,\cdots,x_{n-1}}\xi\rangle_{\frkd}.
\end{align*}
Then
\begin{equation}\label{eq:double-one-g*}
[x_1,\cdots,x_{n-1},\xi]_\frkd=\ad^\ast_{x_1,\cdots,x_{n-1}}\xi.
\end{equation}
Similarly,  for all $\xi_1,\cdots,\xi_n\in \g^*,x\in\g$, we have
\begin{align*}
[\xi_1,\cdots,\xi_{n-1},x]_\frkd &=\ada^\ast_{\xi_1,\cdots,\xi_{n-1}}x,\\
[x_1,\cdots,x_{n-2},\xi_1,\xi_2]_\frkd=0,&\cdots,[x_1,x_2,\xi_1,\cdots,\xi_{n-2}]_\frkd=0.
\end{align*}
Then we get \eqref{eq:$n$-Lie-on-d}.
By the similar discussion to Proposition \ref{pro:double-Lie}, $\gamma_\g$ satisfies \eqref{eq:n-centroid}, \eqref{eq:1} and \eqref{eq:2}. It follows that $^t\gamma_\g: \otimes^n\g^*\rightarrow \g^*$ defines an $n$-Lie algebra structure on $\g^*$, and $(\g,\gamma_\g)$ is a double construction $n$-Lie bialgebra.  This completes the proof.


%

\end{proof}

{\bf Acknowledgements:}
The fourth  author acknowledges support from the NSF China (12101328).

\end{document}